\documentclass[12pt]{amsart}

\usepackage{amsmath,amsthm,amsfonts,amssymb}

\usepackage{graphicx,epstopdf,hyperref}

\usepackage{scalerel}

\newtheorem{theorem}{Theorem}[section]
\newtheorem{lemma}[theorem]{Lemma}
\newtheorem{proposition}[theorem]{Proposition}

\theoremstyle{definition} \newtheorem{definition}[theorem]{Definition}

\theoremstyle{remark} \newtheorem{remark}[theorem]{Remark}

\newcommand{\Q}{\mathbb{Q}} 
\newcommand{\Z}{\mathbb{Z}}

\newcommand{\Ib}{\mathbf{I}}

\newcommand{\A}{\mathcal{A}}
\newcommand{\B}{\mathcal{B}}

\newcommand{\Ccl}{\mathcal{C}_l}

\newcommand{\Nc}{\mathcal{N}}
\newcommand{\Pc}{\mathcal{P}}
\newcommand{\R}{\mathcal{R}}
\newcommand{\Sc}{\mathcal{S}}

\newcommand{\Uq}{\mathcal{U}_q(\mathfrak{gl}(1|1))}
\newcommand{\Uqtwo}{\mathcal{U}_q(\mathfrak{sl}(2))}

\newcommand{\HFK}{\widehat{HFK}}

\newcommand{\HF}{\widehat{HF}}

\DeclareMathOperator{\id}{id}

\DeclareMathOperator{\Rest}{Rest}
\DeclareMathOperator{\Induct}{Induct}

\title[Khovanov-Seidel quiver algebras and Ozsv{\'a}th-Szab{\'o}'s bordered theory]{Khovanov-Seidel quiver algebras and Ozsv{\'a}th-Szab{\'o}'s bordered theory}
\author{Andrew Manion}

\calclayout

\begin{document}

\begin{abstract} We investigate a relationship between Ozsv{\'a}th and Szab{\'o}'s bordered theory \cite{OSzNew} and the algebras and bimodules constructed by Khovanov and Seidel in \cite{KSQuiver}. Specifically we show that (a variant of) a special case of Ozsv{\'a}th-Szab{\'o}'s algebras has a quotient which is isomorphic to the Khovanov-Seidel quiver algebra with coefficients in $\Z/2\Z$. Furthermore, we show that after induction and restriction of scalars, the dg bimodule over quiver algebras associated to a crossing by Khovanov-Seidel is homotopy equivalent to Ozsv{\'a}th-Szab{\'o}'s DA bimodule for the crossing in this special case.
\end{abstract}

\maketitle

\section{Introduction}

In \cite{OSzNew}, Ozsv{\'a}th and Szab{\'o} introduce a family of differential graded algebras $\B(m,k,\Sc)$, where $0 \leq k \leq m$ and $\Sc \subset \{1,\ldots,m\}$. They also define Type D structures, DA bimodules, and Type A structures over these algebras (for a review of this terminology from bordered Heegaard Floer homology, see \cite{LOTBimod}). These algebraic constructions may be used to efficiently compute the knot Floer homology of a knot (\cite{HFKOrig,RThesis}), given a planar projection. 

When $\Sc = \emptyset$, $\B(m,k,\Sc)$ is concentrated in homological degree zero and has no differential. Thus, it may be viewed as an ordinary (noncommutative) algebra with an intrinsic grading by $\Q^m$ called the multi-Alexander grading. In this paper we will be concerned with the special case $\Sc = \emptyset$ and $k = 1$. We will also refine the multi-Alexander grading to a grading by $\Q^{2m}$; see Section~\ref{OSzSect}. It will be more natural to consider a subalgebra $\Ccl(m,k,\Sc)$ of $\B(m,k,\Sc)$, as defined below. 

In \cite{KSQuiver}, Khovanov and Seidel introduce graded rings $A_m$, $m \geq 2$, and dg bimodules $R_i$ over these rings. (Similar rings and bimodules were defined independently in \cite{RouZim}.) The goal of this paper is to relate $A_{m-1} \otimes \Z/2\Z$, and the corresponding dg bimodules, with a $\Z$-graded version $\Ccl^{bot.gr.}(m,1,\emptyset)$ of $\Ccl(m,1,\emptyset)$ and the corresponding Type DA bimodules (the notation is an abbreviation for ``bottom gradings'').

For the algebras, we have the following theorem:
\begin{theorem}[cf. Theorem~\ref{BodyAlgThm}]\label{IntroAlgThm}
Let $m \geq 2$. The Khovanov-Seidel quiver algebra $A_{m-1} \otimes \Z/2\Z$ is isomorphic, as a $\Z$-graded algebra over $\Z/2\Z$, to $\Ccl^{bot.gr.}(m,1,\emptyset)$ modulo the ideal
\[
\langle U_1 + \ldots + U_m \rangle \subset \Ccl^{bot.gr.}(m,1,\emptyset).
\]
\end{theorem}

To a positive generator $\sigma_i$ of the braid group on $m$ strands, Khovanov and Seidel associate a dg bimodule $R_i$ over $(A_{m-1},A_{m-1})$. In Section \ref{DAReformSect}, we define a DA bimodule $\R_i$ with trivial higher actions, whose associated dg bimodule is $R_i$ with the homological grading reversed. 

Similarly, to a negative braid generator $\sigma_i^{-1}$, Khovanov and Seidel associate a dg bimodule $R'_i$, and we will define a DA bimodule $\R'_i$ whose associated dg bimodule is $R'_i$ with the homological grading reversed.

To a positive braid generator $\sigma_i$, Ozsv{\'a}th and Szab{\'o} associate a DA bimodule $\Pc^i$ over $(\B(m,1,\emptyset),\B(m,1,\emptyset))$; we will work with an idempotent-truncated version $\Pc^i_l$ over $(\Ccl(m,1,\emptyset),\Ccl(m,1,\emptyset))$, defined below in Figure \ref{PiGraphFig}.

Similarly, to a negative braid generator $\sigma_i^{-1}$, Ozsv{\'a}th and Szab{\'o} associate a DA bimodule $\Nc^i$ over $(\B(m,1,\emptyset),\B(m,1,\emptyset))$. We will work instead with the truncated version $\Nc^i_l$, defined below in Figure \ref{NiGraphFig}.

The gradings we define on $\Pc^i_l$ and $\Nc^i_l$ are a refinement of Ozsv{\'a}th-Szab{\'o}'s multi-Alexander grading (up to a shift) from which gradings corresponding to those on $\R_i$ and $\R'_i$ can be extracted. When we refer to $\Pc^i_l$ or $\Nc^i_l$, we will always mean either the version with the refined gradings, or the grading-collapsed versions over $(\Ccl^{bot.gr.}(m,1,\emptyset),\Ccl^{bot.gr.}(m,1,\emptyset))$; which one should be clear from context.

Let
\[
\phi: \Ccl^{bot.gr.}(m,1,\emptyset) \to A_{m-1} \otimes \Z/2\Z
\]
denote the quotient map of Theorem~\ref{IntroAlgThm}. Following Section 2.4.2 of \cite{LOTBimod}, one may use $\phi$ to obtain DA bimodules
\[
\Rest_{\phi} \R_i, \Rest_{\phi} \R'_i
\]
and
\[
{^{\phi}}\Induct \Pc^i_l, {^{\phi}}\Induct \Nc^i_l
\]
over $(A_{m-1} \otimes \Z/2\Z,\Ccl^{bot.gr.}(m,1,\emptyset))$.

\begin{theorem}[cf. Theorem~\ref{BodyBimodThmP}, Theorem~\ref{BodyBimodThmN}]\label{IntroBimodThm}
There exist homotopy equivalences
\[
\Rest_{\phi} \R_i \sim {^{\phi}}\Induct \Pc^i_l
\]
and
\[
\Rest_{\phi} \R'_i \sim {^{\phi}}\Induct \Nc^i_l
\]
 of DA bimodules over $(A_{m-1} \otimes \Z/2\Z,\Ccl^{bot.gr.}(m,1,\emptyset))$. 
\end{theorem}

\subsection{Motivation}

An interesting problem in knot theory is to find relationships between the knot homologies defined by Khovanov \cite{KhOrig} and Khovanov-Rozansky \cite{KhRozI, KhRozII} and Heegaard Floer invariants, including knot Floer homology and the Heegaard Floer homology of the branched double cover. Ozsv{\'a}th and Szab{\'o} \cite{OSzBDC} found the first such relationship in the form of a spectral sequence for a link $L$ in $S^3$ with $E_2$ page given by the reduced Khovanov homology of $L$ and $E_{\infty}$ page given by $\HF(\Sigma(m(L)))$, where $\Sigma(m(L))$ is the double cover of $S^3$ branched along the mirror $m(L)$ of $L$. Rasmussen \cite{KPKH} conjectured that for a knot $K$, a spectral sequence exists whose $E_2$ page is the reduced Khovanov homology of $K$ and whose $E_{\infty}$ page is $\HFK(S^3,K)$. Dunfield-Gukov-Rasmussen \cite{DGR} and Rasmussen \cite{RDiff} formulated a similar conjecture involving Khovanov-Rozansky's HOMFLY-PT homology (\cite{KhRozII}); the $E_2$ page should be HOMFLY-PT homology, and the $E_{\infty}$ page should be $\HFK$. There has been recent progress on both conjectures; see \cite{BLS,MUntwisted,GFramedGraphs,Nate15a,Nate15b}.

Roberts \cite{RobertsSS} generalized Ozsv{\'a}th-Szab{\'o}'s spectral sequence in the setting of a link $L$ embedded in a thickened annulus with ``axis'' $B$ (an unknot in $S^3$ linking the standard thickened annulus once). The $E_2$ page of Roberts' sequence is given by the sutured annular Khovanov homology $SKh(L)$ (see \cite{APS1}), and the $E_{\infty}$ page is given by $\HFK(\Sigma(m(L)),\widetilde{B})$ where $\widetilde{B}$ is the preimage of the axis $B$ under the double-covering projection $\Sigma(m(L)) \to S^3$. 

Given a closed braid $L$ which is the closure of an open braid $L'$, one can try to study invariants of $L$ or $L'$ ``locally,'' letter by letter in a braid word. In \cite{AGW13}, Auroux, Grigsby, and Wehrli analyze $SKh$ along these lines: they show that one summand of $SKh(L)$ is the Hochschild homology of the Khovanov-Seidel dg bimodule associated to $m(L')$ by tensoring together dg bimodules $R_i$ and $R'_i$ for braid generators $\sigma_i$ and $\sigma_i^{-1}$. They also obtain a ``local'' version of the branched-double-cover spectral sequence; see \cite{AGW11} too.

The Khovanov-Seidel quiver algebras $A_{m-1}$ and dg bimodules are the special case $k = 1$ of graded algebras $A^{k,m-k}$ and dg $(A^{k,m-k},A^{k,m-k})$-bimodules. These algebras and bimodules were defined explicitly by Stroppel \cite{Str09} (in the case $m = 2k$, but see Remark 5.8.2 of \cite{Str09}) and Brundan-Stroppel \cite{BrStr1, BrStr2}. They were defined independently by Chen-Khovanov \cite{ChenKh}. The authors of \cite{AGW13} conjectured that the rest of the summands of $SKh(L)$ are isomorphic to the Hochschild homology of the appropriate dg bimodules over $(A^{k,m-k},A^{k,m-k})$; this conjecture has recently been proved in \cite{BPW}.

Whereas the relationship between bimodules over $(A^{k,m-k},A^{k,m-k})$ and sutured annular Khovanov homology is relatively new, the relationship with the usual Khovanov homology has been established since the definition of these algebras and bimodules and was an important motivation for their introduction; see \cite{ChenKh, Str09}. Rather than taking Hochschild homology, one introduces additional dg bimodules for local maxima and minima of a link projection. The tensor product of all dg bimodules for a closed link is a chain complex whose homology is Khovanov homology (up to mirroring conventions which we will ignore).

The sum of the algebras $A^{k,m-k}$, for $0 \leq k \leq m$, can be viewed as a categorification of the representation $V^{\otimes m}$ of $\Uqtwo$, where $V$ is the fundamental representation. Each individual algebra $A^{k,m-k}$ categorifies a different weight space of $V^{\otimes m}$. The existence of dg bimodules for each local piece of a link projection, such that the tensor product over all pieces of a closed link computes Khovanov homology, is a categorification of the existence of a ribbon structure on the $\otimes$-subcategory of $\Uqtwo$-representations generated by $V$ which encodes the Jones polynomial. (The whole category of finite-dimensional representations of $\Uqtwo$ also has a ribbon structure, which encodes the colored Jones polynomials for $\mathfrak{sl}(2)$.)

From the above, it is clear that the Khovanov-Seidel quiver algebras and dg bimodules are closely related to Khovanov homology, as well as to related variants like the sutured annular theory. Using the framework of the Ozsv{\'a}th-Szab{\'o} spectral sequence, the algebras and bimodules have also been shown to be related to Heegaard-Floer invariants of the branched double cover in \cite{AGW13,AGW11}. The current paper explores yet another appearance of Khovanov-Seidel's construction, this time in connection with a theory that computes knot Floer homology.

Whereas the Jones polynomial can be interpreted using representations of $\Uqtwo$, the Alexander polynomial can be interpreted using representations of $\Uq$ (see \cite{RozSal, Resh92, KSFF} as well as \cite{VirQ, SartoriAlexander} for more modern treatments). It is natural to ask, then, whether Ozsv{\'a}th-Szab{\'o}'s theory in \cite{OSzNew} for knot Floer homology can be viewed as categorifying some $\Uq$-representation setup. This is indeed the case and will be addressed in  \cite{MyDecatPaper}.

Sartori \cite{Sartori} has a categorification of the representation $V^{\otimes m}$ of $\Uq$, where $V$ is the fundamental representation, as well as a categorification of the action of braid group generators on $V^{\otimes m}$. It is hoped that a suitable generalization of Sartori's theory will be shown to compute knot Floer homology. The Khovanov-Seidel quiver algebras and dg bimodules appear as the restriction of Sartori's categorification to one particular weight space of $V^{\otimes m}$.

The appearance of Khovanov-Seidel quiver algebras in categorifications of both $\Uqtwo$-representations and $\Uq$-representations is an interesting coincidence. On the decategorified level, it can be explained as follows: indecomposable projective modules over the categorification algebras correspond to certain canonical or dual-canonical basis elements for $V^{\otimes m}$. The decomposition of $V^{\otimes m}$ into weight spaces is the same whether $V^{\otimes m}$ is regarded as a $\Uqtwo$-module or a $\Uq$-module, but the canonical bases are different. However, the bases do agree (with suitable conventions) on one particular weight space, and this weight space is the one categorified by the Khovanov-Seidel quiver algebras.

The braid group action on this particular weight space of $V^{\otimes m}$ may be identified with the Burau representation of the braid group. This representation has a natural topological definition. Its relationship with the Alexander polynomial is classical; the connection with $\Uq$-representations was introduced in \cite{KSFF}. The relationship with the Jones polynomial appears already in Jones' paper \cite{JonesOrig}.

Generalizing the result of this paper to find a spectral sequence between Khovanov homology and knot Floer homology seems difficult, since the coincidence discussed above does not hold for most of the weight spaces of $V^{\otimes m}$. One possible way forward would be to find categorifications of $V^{\otimes m}$ as a module over $\Uqtwo$ and $\Uq$ such that the indecomposable projective modules over the categorification algebras are in bijection with standard tensor-product basis elements of $V^{\otimes m}$, rather than canonical basis elements. This tensor product basis does not depend on the structure of $V^{\otimes m}$ as a representation. On the $\Uq$ side, Tian \cite{TianUT} categorifies the tensor product basis of $V^{\otimes m}$, although he works in a modified setting. For both $\Uqtwo$ and $\Uq$, unpublished work of Rouquier provides a promising candidate for the tensor-product-basis categorification. The hope is that this construction, applied to $\Uq$, can be shown to compute knot Floer homology, and that viewing the $\Uq$ categorification as a deformation of the $\Uqtwo$ categorification will yield a spectral sequence from Khovanov homology to knot Floer homology.

\subsection*{Acknowledgments}
The author would especially like to thank Robert Lipshitz, Ciprian Manolescu, Rapha{\"e}l Rouquier, Antonio Sartori, and Zolt{\'a}n Szab{\'o} for useful suggestions, comments, and discussions concerning this paper. The author was supported through NSF grants DGE-1148900 and DMS-1502686.

\section{Khovanov-Seidel's construction}\label{KSSect}

We briefly review some definitions from \cite{KSQuiver}. 

\begin{definition}[cf. \cite{KSQuiver}, Section 1b]
Let $m \geq 2$. The graded ring (or graded $\Z$-algebra) $A_{m-1}$ is the quotient of the quiver algebra
\[
(0) \mathrel{\mathop{\rightleftarrows}^{(0|1)}_{(1|0)}} (1) \mathrel{\mathop{\rightleftarrows}^{(1|2)}_{(2|1)}} (2) \mathrel{\mathop{\rightleftarrows}^{(2|3)}_{(3|2)}} \cdots \mathrel{\mathop{\rightleftarrows}^{(m-2|m-1)}_{(m-1|m-2)}} (m-1),
\]
with $\deg (i) = \deg (i|i+1) = 0$ and $\deg (i+1|i) = 1$, by the relations 
\begin{align*}
&(i-1|i|i+1) = 0 = (i+1|i|i-1), \\
& (i|i+1|i) = (i|i-1|i), \\
&(0|1|0) = 0
\end{align*}
for $0 < i < m-1$. The notation means, for example, 
\[
(i-1|i|i+1) := (i-1|i) \cdot (i|i+1).
\]
\end{definition}

\begin{remark}
In \cite{KSQuiver}, Khovanov and Seidel write composition in a quiver algebra left-to-right rather than right-to-left; thus, for example, the algebra element $(i|i+1)$ is an arrow from vertex $(i)$ to vertex $(i+1)$, not the other way around. We choose to follow this convention throughout the paper.
\end{remark}

For $0 \leq i \leq m-1$, the elementary idempotent of $A_{m-1}$ corresponding to vertex $(i)$ of the quiver is also denoted $(i)$. The subalgebra of $A_{m-1}$ generated by the idempotents $(i)$ will be denoted $\Ib_l(m)$; it is isomorphic to $\Pi_{i=0}^{m-1} \Z$ as a ring. When working modulo $2$, we will abuse notation and identify $\Ib_l(m)$ with the $\Z/2\Z$-subalgebra of $A_{m-1} \otimes \Z/2\Z$ generated by $\{(i): 0 \leq i \leq m-1\}$, which is isomorphic to $\Pi_{i=0}^{m-1} \Z/2\Z$ as a $\Z/2\Z$-algebra. We may view $A_{m-1}$ as an algebra over its idempotent ring $\Ib_l(m)$.

For each elementary idempotent $(i)$, we have an indecomposable projective left $A_{m-1}$-module
\[
P_i := A_{m-1}(i)
\]
and an indecomposable projective right $A_{m-1}$-module
\[
{_i}P := (i)A_{m-1}.
\]

The definition of Khovanov-Seidel's dg bimodules $R_i$ and $R'_i$ uses morphisms
\[
\beta_i: P_i \otimes_{\Z} {_i}P \to A_{m-1}
\]
and
\[
\gamma_i: A_{m-1} \to P_i \otimes_{\Z} {_i}P \{-1\}
\]
of $(A_{m-1},A_{m-1})$-bimodules, where $\{ \cdot \}$ denotes upward shift in the intrinsic grading (following the conventions of \cite{KSQuiver}). These morphisms are defined by 
\[
\beta_i((i) \otimes (i)) := (i)
\]
and
\begin{align*}
\gamma_i(1) &:= (i-1 | i) \otimes (i | i-1) \\
&+ (i+1|i) \otimes (i|i+1) \\
&+ (i) \otimes (i|i-1|i) \\
&+ (i|i-1|i) \otimes (i).
\end{align*}

\begin{definition}[cf. Section 2d of \cite{KSQuiver}] 
For $1 \leq i \leq m-1$, define $R_i$ to be the complex of $A_{m-1}$-modules
\[
0 \rightarrow P_i \otimes_{\Z} {_i}P \xrightarrow{\beta_i} A_{m-1} \rightarrow 0.
\]
The intrinsic grading on $R_i$ is inherited from $P_i \otimes {_i}P$ and $A_{m-1}$. $A_{m-1}$ is placed in homological degree $0$, and $P_i \otimes {_i}P$ is placed in homological degree $-1$.

Define $R'_i$ to be the complex of $A_{m-1}$-modules
\[
0 \rightarrow A_{m-1} \xrightarrow{\gamma_i} P_i \otimes_{\Z} {_i}P \{-1\} \rightarrow 0.
\]
$A_{m-1}$ is placed in homological degree $0$, and $P_i \otimes_{\Z} {_i}P \{-1\}$ is placed in homological degree $1$.
\end{definition}

\section{Reformulation with DA bimodules}\label{DAReformSect}

Ozsv{\'a}th and Szab{\'o} assign Type DA bimodules to crossings in Section 5 of \cite{OSzNew}. Thus, it will be useful to view Khovanov-Seidel's dg bimodules $R_i$ in this context. 

Let $\A$, $\A'$ be dg algebras over idempotent rings $\Ib$, $\Ib'$ (the idempotent rings should be finite direct products of copies of $\Z/2\Z$). A DA bimodule 
\[
X = {^{\A}}X_{\A'}
\] 
over $(\A,\A')$ is a (left, right) bimodule over $(\Ib,\Ib')$ together with, for each $j \geq 1$, a grading-preserving $(\Ib,\Ib')$-bilinear map
\[
\delta^1_j: X \otimes_{\Ib'} T^{j-1} (\A'[1]) \to \A[1] \otimes_{\Ib} X,
\]
satisfying certain compatibility conditions. The shift $[\cdot]$ is an upward shift in the homological grading. For a complete definition and some basic properties of DA bimodules, see Section 2.2.4 of \cite{LOTBimod}.

The dg algebras we consider all have a homological grading by $\Z$ as well as an intrinsic grading by $\Z$, $\Q^m$, $\Q^{2m}$, or a similar group. Our DA bimodules will also have a homological grading by $\Z$ and an intrinsic grading by a grading group $G$, together with homomorphisms from the grading groups of the two algebras into $G$. For the DA bimodules considered here, $G$ will be the same as the grading groups of the two algebras, but the homomorphisms will not necessarily be the identity. 

If $X$ is a DA bimodule over $(\A,\A')$ with $\delta^1_j = 0$ for $j \geq 3$, then we have an honest dg bimodule $\A \boxtimes X$ over $(\A,\A')$, graded by $\Z \oplus G$ (see \cite{LOTBimod} for the definition of $\boxtimes$; in this case, $\A \boxtimes X := \A \otimes_{\Ib} X$). Left multiplication by $\A$ is defined as multiplication in the $\A$-factor of $\A \boxtimes X$. Right multiplication by $\A'$ is defined using the DA operation $\delta^1_2$, and the differential is defined using $\delta^1_1$. The DA compatibility conditions ensure that $\A \boxtimes X$ is a well-defined dg bimodule. Since $\delta^1_1$ preserves gradings, the differential on $\A \boxtimes X$ decreases homological grading by $1$.

\begin{definition}\label{KSDADef}
As a bimodule over the idempotent ring $\Ib_l(m)$, the DA bimodule $\R_i$ over $(A_{m-1} \otimes \Z/2\Z, A_{m-1} \otimes \Z/2\Z)$ is defined as 
\[
((i) \otimes_{\Z/2\Z} {_i}P) \oplus \Ib_l(m),
\]
with intrinsic $\Z$-grading inherited from 
\[
((i) \otimes_{\Z/2\Z} {_i}P) = {_i}P.
\]
$\Ib_l(m)$ is placed in intrinsic degree $0$. The DA operations are given by
\[
\delta^1_1 \langle (i) \otimes {_i}a_j \rangle = a \otimes \langle (j) \rangle,
\]
\[
\delta^1_2(\langle (i) \otimes {_i}a_j \rangle \otimes {_j}b_k) = \langle (i) \otimes ab \rangle,
\]
and
\[
\delta^1_2(\langle (i) \rangle \otimes {_i}a_j) = a \otimes \langle (j) \rangle.
\]
We use angle brackets $\langle \cdot \rangle$ around DA bimodule generators here to avoid confusion. All higher DA operations are zero. The homological degree of $\Ib_l(m)$ is defined to be $0$. The homological degree of $((i) \otimes_{\Z/2\Z} {_i}P)$ is defined to be $1$.
\end{definition}

\begin{definition}\label{KSDAPrimeDef}
As a bimodule over $\Ib_l(m)$, the DA bimodule $\R'_i$ over $(A_{m-1} \otimes \Z/2\Z, A_{m-1} \otimes \Z/2\Z)$ is defined as 
\[
\Ib_l(m) \oplus ((i) \otimes_{\Z/2\Z} {_i}P)\{-1\},
\]
with intrinsic $\Z$-grading inherited from 
\[
((i) \otimes_{\Z/2\Z} {_i}P)\{-1\} = {_i}P\{-1\}.
\]
$\Ib_l(m)$ is placed in intrinsic degree $0$. The DA operations are given by
\[
\delta^1_1 \langle (i-1)  \rangle = (i-1|i) \otimes \langle (i) \otimes (i|i-1) \rangle,
\]
\[
\delta^1_1 \langle (i+1)  \rangle = (i+1|i) \otimes \langle (i) \otimes (i|i+1) \rangle,
\]
\[
\delta^1_1 \langle (i)  \rangle = (i|i-1|i) \otimes \langle (i) \otimes (i) \rangle + 1 \otimes \langle (i) \otimes (i|i-1|i) \rangle,
\]
\[
\delta^1_2(\langle (i) \otimes {_i}a_j \rangle \otimes {_j}b_k) = \langle (i) \otimes ab \rangle,
\]
and
\[
\delta^1_2(\langle (i) \rangle \otimes {_i}a_j) = a \otimes \langle (j) \rangle.
\]
All higher DA operations are zero. The homological degree of $\Ib_l(m)$ is defined to be $0$. The homological degree of $((i) \otimes_{\Z/2\Z} {_i}P)\{-1\}$ is defined to be $-1$.
\end{definition}

\begin{remark}
With Khovanov-Seidel's conventions, differentials increase homological grading by $1$. Thus, to relate Khovanov-Seidel's and Ozsv{\'a}th-Szab{\'o}'s constructions, we will need to reverse the homological gradings on one side or the other. We choose the degree $(-1)$ differentials for DA bimodules, following \cite{LOTBimod}; thus, the homological grading on $\R_i$ and $\R'_i$ is the negative of the homological grading on $R_i$ and $R'_i$.
\end{remark}

\begin{figure} \centering
\includegraphics[scale=0.625]{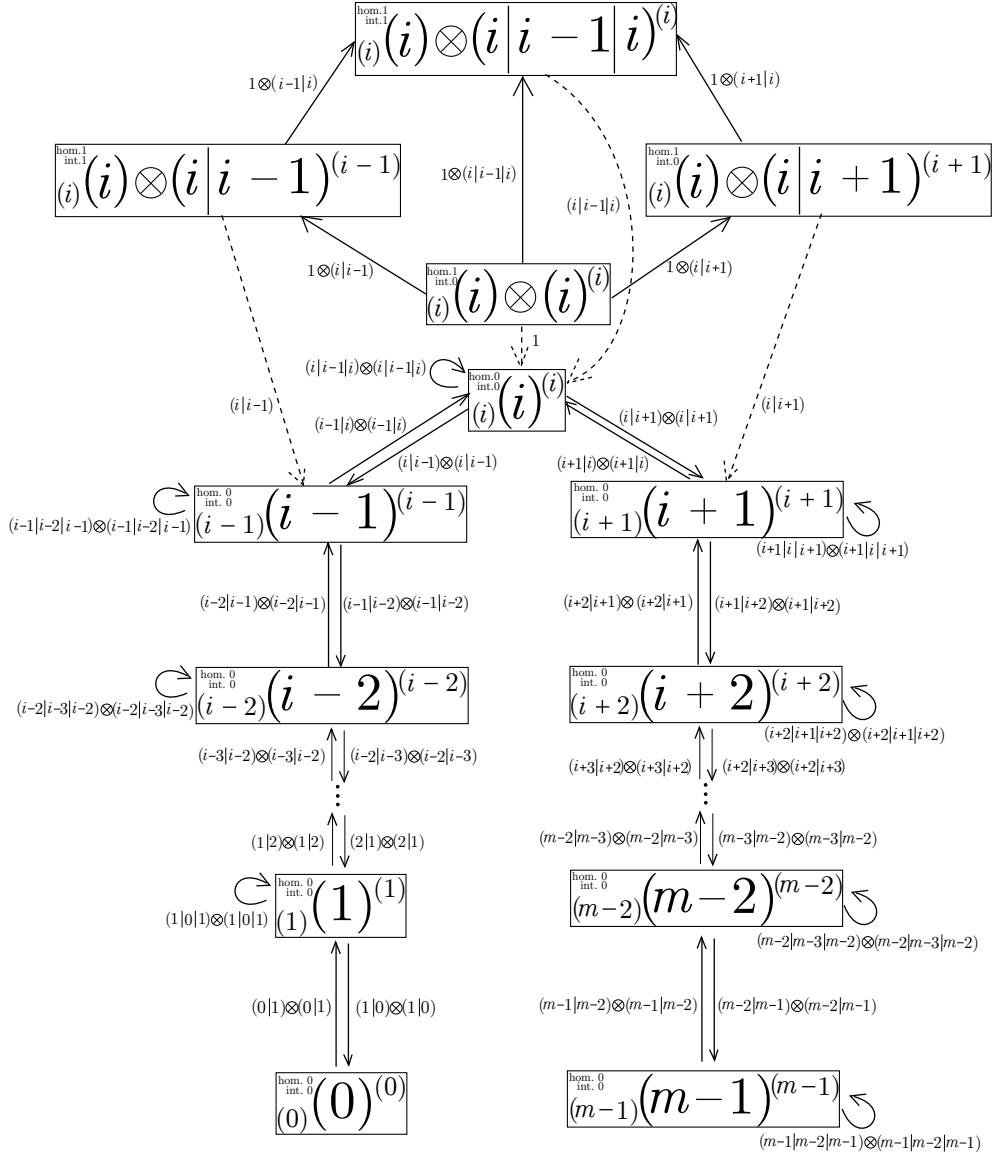}
\caption{The DA bimodule $\R_i$ for $1 \leq i < m-1$.}
\label{RiGraphFig}
\end{figure}

\begin{figure} \centering
\includegraphics[scale=0.625]{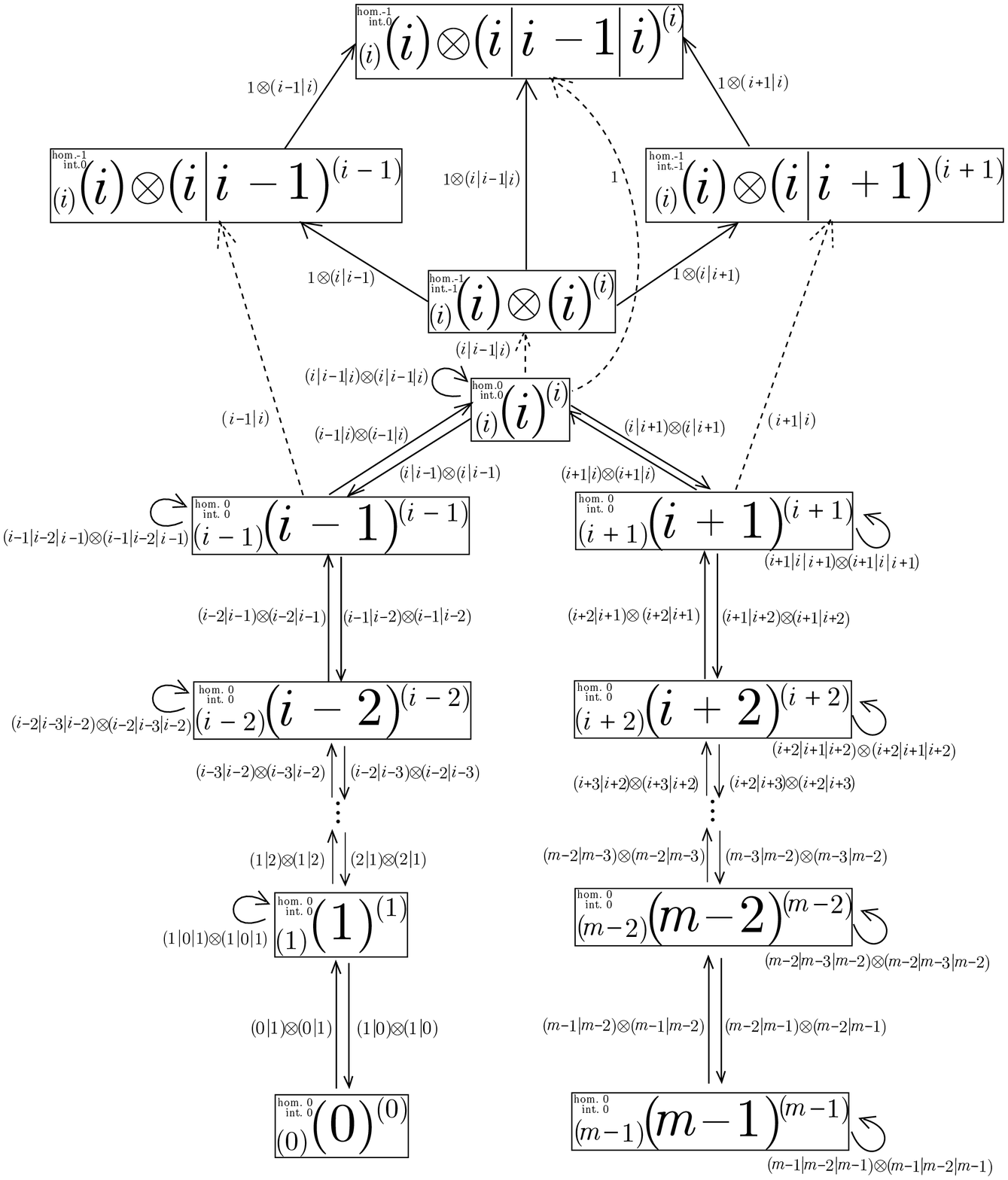}
\caption{The DA bimodule $\R'_i$ for $1 \leq i < m-1$. The solid arrows are the same as in Figure~\ref{RiGraphFig}; only the dotted arrows and gradings of generators have changed.}
\label{RPrimeiGraphFig}
\end{figure}

\begin{proposition}
$R_i \otimes \Z/2\Z$ is isomorphic to the dg bimodule 
\[
(A_{m-1} \otimes \Z/2\Z) \boxtimes \R_i
\]
associated to $\R_i$, via an isomorphism which preserves intrinsic degrees and multiplies homological degrees by $-1$. Similarly, $R'_i \otimes \Z/2\Z$ is isomorphic to 
\[
(A_{m-1} \otimes \Z/2\Z) \boxtimes \R'_i,
\]
via an isomorphism which preserves intrinsic degrees and reverses homological degrees.
\end{proposition}

\begin{proof}
This follows immediately from the definitions.
\end{proof}

In Figure~\ref{RiGraphFig}, we depict the DA bimodule $\R_i$ (for $1 \leq i < m-1$) in the notation of Figure 5.4 of \cite{OSzNew}. Generators are shown in rectangles; the lower left and upper right subscripts indicate which idempotents can multiply these generators nontrivially on the left and right. A dotted arrow from a generator $x$ to another generator $y$, labeled by an algebra element $a$, indicates that $a \otimes y$ is a nonzero term of $\delta^1_1(x)$. 

A solid arrow from $x$ to $y$, labeled by an expression $a \otimes b$, indicates that $a \otimes y$ is a nonzero term of $\delta^1_2(x \otimes b)$. If $a$ or $b$ contain sums, we implicitly expand them out.

In general, a solid arrow from $x$ to $y$ labeled by $a \otimes (b_1, \ldots, b_{j-1})$ indicates that $a \otimes y$ is a nonzero term of $\delta^1_j(x \otimes (b_1 \otimes \cdots \otimes b_{j-1}))$ for $j > 1$. All sums should be expanded out. For $\R_i$ and $\R'_i$, none of these higher-action arrows are needed; however, we will need arrows labeled by $a \otimes (b_1,b_2)$ representing $\delta^1_3$ actions when we discuss Ozsv{\'a}th-Szab{\'o}'s bimodules $\Pc^i_l$ and $\Nc^i_l$.

Finally, for each generator $x$, there should be a solid arrow from $x$ to $x$ labeled by $1 \otimes 1$. To save space, we omit these arrows from the diagrams. Figure~\ref{RPrimeiGraphFig} shows $\R'_i$ in the same notation.

\begin{figure} \centering
\includegraphics[scale=0.625]{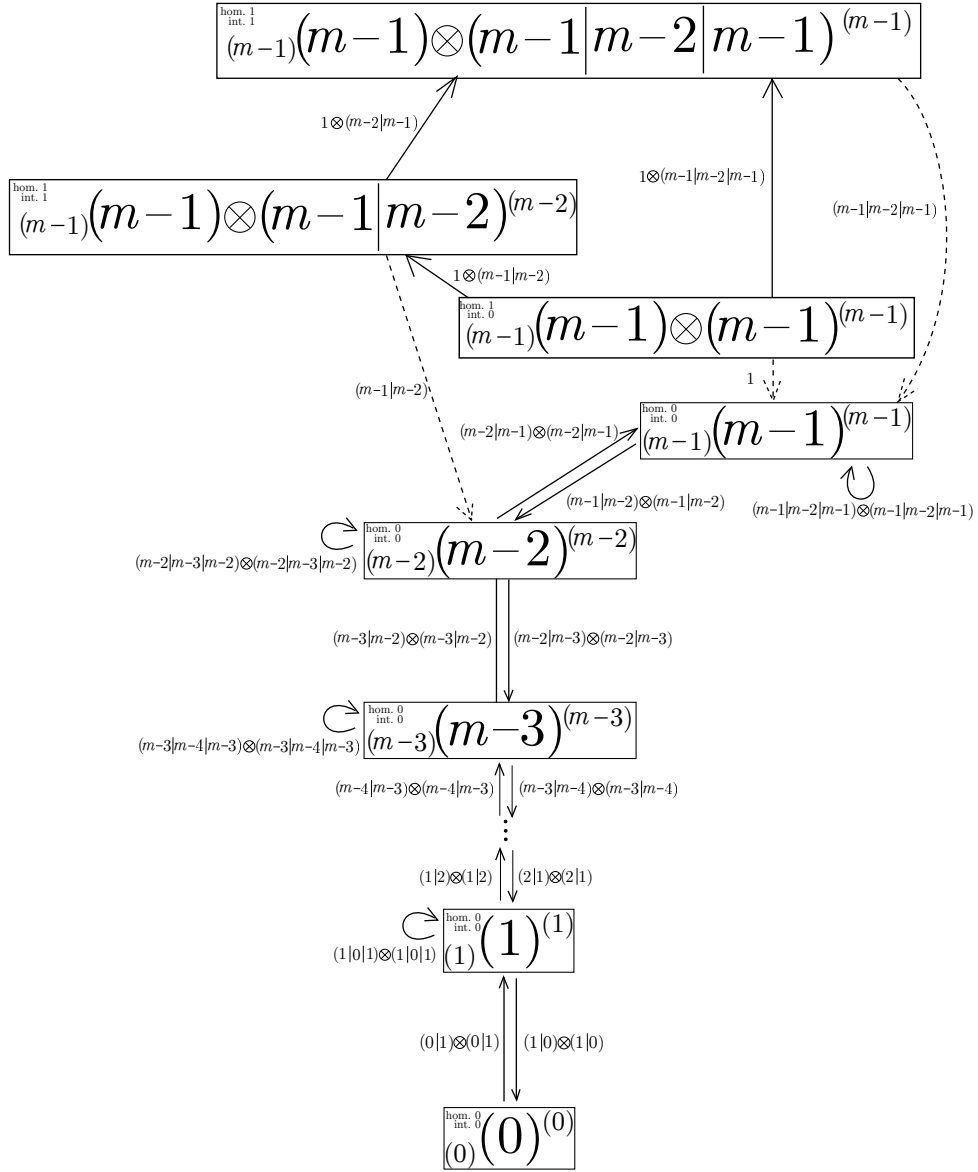}
\caption{The DA bimodule $\R_{m-1}$. The bimodule $\R'_{m-1}$ for a negative crossing is obtained similarly from Figure \ref{RPrimeiGraphFig}.}
\label{SpecialRiGraphFig}
\end{figure}

In Figure~\ref{SpecialRiGraphFig}, we depict the special case $\R_{m-1}$. It is obtained by omitting some generators and the adjacent arrows from the generic diagram for $\R_i$ in Figure~\ref{RiGraphFig}. The bimodule $\R'_{m-1}$ works similarly. 

\section{Ozsv{\'a}th-Szab{\'o}'s construction}\label{OSzSect}

Now we review definitions from \cite{OSzNew} with a grading refinement. We discuss gradings first. Ozsv{\'a}th-Szab{\'o}'s multi-Alexander grading on $\B(m,1,\emptyset)$ takes values in $\Q^m$:

\begin{definition}
Ozsv{\'a}th and Szab{\'o}'s multi-Alexander grading set is
\[
\Q\langle e_1, \ldots, e_m \rangle \cong \Q^{m},
\]
where the $e_i$ are names for basis vectors. 
\end{definition}

The coefficient of $e_i$ counts the ``weight'' of an algebra element on strand $i$. We may actually define a grading by $\Q^{2m}$ which counts the weight on the top and bottom of each strand. This refined grading is natural when thinking of algebra elements as determining limiting behavior of holomorphic curves for the appropriate Heegaard diagram.

\begin{definition}
Let the refined multi-Alexander grading set be 
\[
\Q\langle \tau_1, \ldots, \tau_m, \beta_1, \ldots, \beta_m \rangle \cong \Q^{2m},
\]
where the $\tau_i$ and $\beta_i$ are names for basis vectors. 
\end{definition}

In Section \ref{KSSect}, $\beta_i$ had an unrelated meaning, but since we will be working with the DA version $\R_i$ of $R_i$, we will not need the earlier $\beta_i$ again. There is yet another conflict of notation: in \cite{OSzNew}, $\tau_i^{\mathbf{gr}}$ refers to a homomorphism of grading groups. These conflicts are minor, and hopefully no confusion will arise.

One should think of $\tau_i$ as having weight $1$ on the top of strand $i$, and weight $0$ elsewhere. Similarly, $\beta_i$ has weight $1$ on the bottom of strand $i$ and weight $0$ elsewhere.

From the refined multi-Alexander grading, we can extract Ozsv{\'a}th-Szab{\'o}'s multi-Alexander gradings, as well as a single $\Z$-grading which will correspond to the grading on $A_{m-1}$:
\begin{definition}\label{EtaEpsilonDef}
The group homomorphism
\[
\eta: \Q \langle \tau_1, \ldots, \tau_m, \beta_1, \ldots, \beta_m \rangle \to \Q \langle e_1, \ldots, e_m \rangle
\]
sends each $\tau_i$ and $\beta_i$ to $e_i$. The homomorphism
\[
\epsilon: \frac{1}{2}\Z \langle \tau_1, \ldots, \tau_m, \beta_1, \ldots, \beta_m \rangle \to \Z
\]
sends each $\beta_i$ to $2 \in \Z$ and sends each $\tau_i$ to zero.
\end{definition}

Ozsv{\'a}th-Szab{\'o}'s multi-Alexander grading will be obtained by applying $\eta$ to the refined multi-gradings. Khovanov-Seidel's grading will be obtained by applying $\epsilon$ instead.

\subsection{Algebras}

\begin{definition}[cf. Section 3.2 of \cite{OSzNew}]
Let $m \geq 2$ (to parallel Khovanov-Seidel's definitions). The $\Q^{2m}$-graded $\Z/2\Z$-algebra $\B(m,1,\emptyset)$ is the quotient of the quiver algebra
\[
(0) \mathrel{\mathop{\rightleftarrows}^{R_1}_{L_1}} (1) \mathrel{\mathop{\rightleftarrows}^{R_2}_{L_2}} (2) \mathrel{\mathop{\rightleftarrows}^{R_3}_{L_3}} \cdots \mathrel{\mathop{\rightleftarrows}^{R_{m-1}}_{L_{m-1}}} (m-1) \mathrel{\mathop{\rightleftarrows}^{R_m}_{L_m}} (m) ,
\]
with $\deg R_i = \frac{1}{2} \tau_i$ and $\deg L_i = \frac{1}{2} \beta_i$, by the relations
\[
R_i R_{i+1} = 0 = L_{i+1} L_i
\]
for $1 \leq i \leq m-1$. Note that the refined grading takes values in $(\frac{1}{2}\Z)^{2m} \subset \Q^{2m}$. Applying the homomorphism $\eta$ of Definition~\ref{EtaEpsilonDef} gives Ozsv{\'a}th-Szab{\'o}'s multi-Alexander gradings. Applying $\epsilon$ instead yields $\deg R_i = 0$ and $\deg L_i = 1$.
\end{definition}

For $0 \leq i \leq m$, the elementary idempotent of $\B(m,1,\emptyset)$ corresponding to vertex $(i)$ of the quiver is denoted $\Ib_{(i)}$. For $1 \leq i \leq m$, we may also consider the elements
\[
U_i := R_i L_i + L_i R_i.
\]
The element $U_i$ only multiplies nontrivially with the idempotents $\Ib_{(i-1)}$ and $\Ib_{(i)}$ (compare with relation (3.5) in Section 3.2 of \cite{OSzNew}). The $U_i$ are central, and
\[
U_i^N = (R_i L_i)^N + (L_i R_i)^N \neq 0.
\]
Thus, $\B(m,1,\emptyset)$ is infinite-dimensional over $\Z/2\Z$. The refined multi-Alexander degree of $U_i$ is $\frac{1}{2}(\tau_i + \beta_i)$.

We will be more interested in an idempotent-truncated version of $\B(m,1,\emptyset)$:
\begin{definition}
$\Ccl(m,1,\emptyset)$ is defined to be 
\[
\bigg(\sum_{0 \leq i \leq m-1} \Ib_{(i)} \bigg) \cdot \B(m,1,\emptyset) \cdot \bigg( \sum_{0 \leq i \leq m-1} \Ib_{(i)} \bigg)
\]
\end{definition}

\begin{figure} \centering
\includegraphics[scale=0.325]{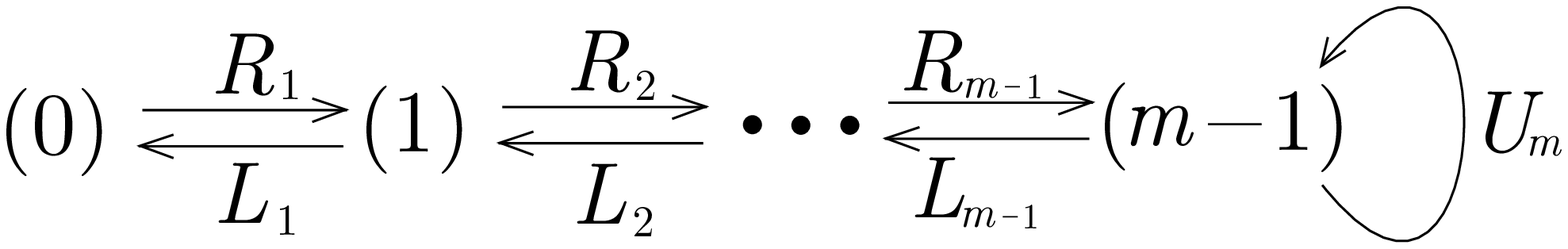}
\caption{Quiver for $\Ccl(m,1,\emptyset)$}
\label{ClQuiverFig}
\end{figure}

$\Ccl(m,1,\emptyset)$ is the algebra of the quiver shown in Figure \ref{ClQuiverFig} modulo the relations $R_i R_{i+1} = 0$ and $L_{i+1} L_i = 0$ for $1 \leq i \leq m-2$, as well as $R_{m-1} U_m = 0$ and $U_m L_{m-1} = 0$.

Since the refined grading on $\Ccl(m,1,\emptyset)$ takes values in $(\frac{1}{2}\Z)^{2m} \subset \Q^{2m}$, the following definition makes sense:
\begin{definition}
The $\Z$-graded $\Z/2\Z$-algebra $\Ccl^{bot.gr.}(m,1,\emptyset)$ is obtained by applying the homomorphism $\epsilon$ to the refined multi-gradings of $\Ccl(m,1,\emptyset)$.
\end{definition}

\begin{proposition}
The map $\phi: \Ccl^{bot.gr.}(m,1,\emptyset) \to A_{m-1} \otimes \Z/2\Z$ defined by sending $\Ib_{(i)}$ to $(i)$, $R_i$ to $(i-1|i)$, $L_i$ to $(i|i-1)$, and $U_m$ to $(m-1|m-2|m-1)$ is a surjective homomorphism of $\Z$-graded algebras.
\end{proposition}

\begin{proof}
This follows by comparing the generators and relations of the two algebras. The relations of $\Ccl^{bot.gr.}(m,1,\emptyset)$ are satisfied in $\A_{m-1} \otimes \Z/2\Z$, and each generator of $\A_{m-1} \otimes \Z/2\Z$ is the image of a generator of $\Ccl^{bot.gr.}(m,1,\emptyset)$ under $\phi$.

For the gradings, $R_i$ has refined multi-grading $\frac{1}{2} \tau_i$, which gets sent to $0$ by $\epsilon$; correspondingly, $(i-1|i)$ has degree $0$. $L_i$ has refined multi-grading $\frac{1}{2} \beta_i$, which gets sent to $1$ by $\epsilon$, correspondingly, $(i|i-1)$ has degree $1$.
\end{proof}

\begin{theorem}[cf. Theorem~\ref{IntroAlgThm}]\label{BodyAlgThm}
The kernel of $\phi$ is the ideal of $\Ccl^{bot.gr.}(m,1,\emptyset)$ generated by $U_1 + \cdots + U_m$. Thus, $A_{m-1} \otimes \Z/2\Z$ is isomorphic to the quotient of $\Ccl^{bot.gr.}(m,1,\emptyset)$ by this ideal.
\end{theorem}

\begin{proof}
We have
\begin{align*}
U_1 + \cdots + U_m &= 1 \cdot (U_1 + \cdots + U_m) \\
&= \bigg( \sum_{i=0}^{m-1} \Ib_{(i)} \bigg) \cdot (U_1 + \cdots + U_m) \\
&= \Ib_{(0)} \cdot U_1 + \sum_{i=1}^{m-1} \Ib_{(i)} \cdot (U_i + U_{i+1}) \\
&= R_1 L_1 + \sum_{i=1}^{m-2} (L_i R_i + R_{i+1} L_{i+1}) + (L_{m-1} R_{m-1} + U_m).
\end{align*}
Since each of the $m$ terms in this last sum have distinct left (and right) idempotents, we may multiply by each of these idempotents in turn to see that $R_1 L_1$,  $L_i R_i + R_{i+1} L_{i+1}$ ($1 \leq i \leq m-2$), and $L_{m-1} R_{m-1} + U_m$ are contained in the ideal $\langle U_1 + \cdots + U_m \rangle$.

Thus,
\begin{equation}\label{SumOfUDecomp}
\langle U_1 + \cdots + U_m \rangle = \langle R_1 L_1, L_i R_i + R_{i+1} L_{i+1}, L_{m-1} R_{m-1} + U_m \rangle_{1 \leq i \leq m-2}.
\end{equation}

Each generator of the ideal on the right is in the kernel of $\phi$. Conversely, suppose $a \in \Ccl^{bot.gr.}(m,1,\emptyset)$ is in the kernel of $\phi$. By adding an element of the ideal generated by $L_{m-1} R_{m-1} + U_m$, we may assume $a$ is a sum of products of generators $L_i$ and $R_i$, with all instances of $U_m$ removed. In this case $a$ is a sum of paths in the Khovanov-Seidel quiver which can be expressed as a sum of multiples of $(i-1|i|i+1)$, $(i+1|i|i-1)$, $(i|i+1|i) + (i|i-1|i)$, and $(0|1|0)$ for $1 \leq i \leq m-2$. Comparing with the relations defining $\Ccl^{bot.gr.}(m,1,\emptyset)$, we see that $a$ is in the ideal on the right side of Equation \ref{SumOfUDecomp}.
\end{proof}

\begin{remark}
Note that $A_{m-1} \otimes \Z/2\Z$ is finite-dimensional. Thus, the powers of $U_i$ which make $\Ccl^{bot.gr.}(m,1,\emptyset)$ infinite-dimensional must eventually be sent to zero under $\phi$. Indeed, for each $i$, $U_i^2$ is already in the ideal $\langle U_1 + \cdots + U_m \rangle$. 
\end{remark}

\subsection{DA bimodules}

Now we review Ozsv{\'a}th and Szab{\'o}'s DA bimodules $\Pc^i$ and $\Nc^i$ associated to positive and negative crossings between strands $i$ and $i+1$. Actually, we will focus on truncated versions $\Pc^i_l$ and $\Nc^i_l$, and we will only consider the case relevant to this paper ($k = 1, \Sc = \emptyset$).

Furthermore, we will define a refined multi-Alexander grading on $\Pc^i_l$ and $\Nc^i_l$ taking values in the grading group $\Q^{2m} = \Q \langle \tau_1, \ldots, \beta_m \rangle$. All generators of $\Pc^i_l$ and $\Nc^i_l$ are labeled $N$, $W$, $E$, or $S$, for ``north,'' ``west,'' ``east,'' or ``south,'' and the refined grading depends only on this label. 

\begin{figure} \centering
\includegraphics[scale=0.625]{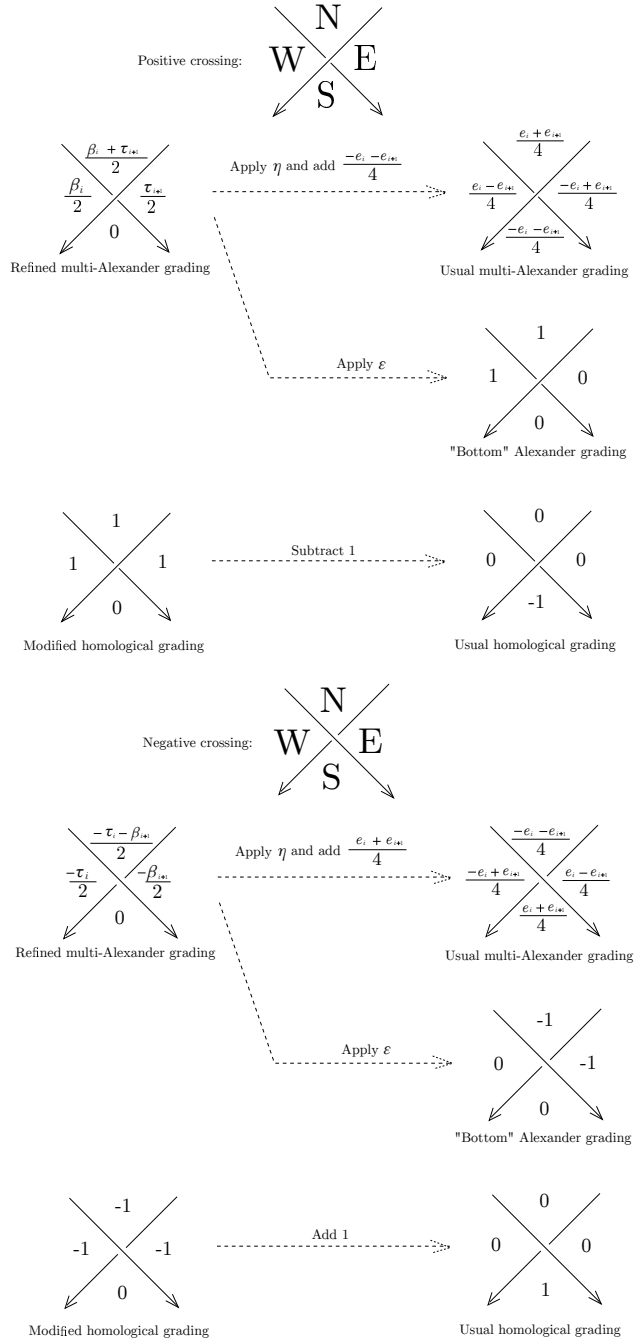}
\caption{Summary of gradings.}
\label{GradingFig}
\end{figure}

\begin{definition}\label{RefinedGradDef}
The refined multi-Alexander gradings of $N$, $W$, $E$, and $S$ are shown in Figure \ref{GradingFig}; the maps $\eta$ and $\epsilon$ are from Definition~\ref{EtaEpsilonDef}. Ozsv{\'a}th-Szab{\'o}'s multi-Alexander grading for a positive crossing is $\frac{-e_i - e_{i+1}}{4}$ plus $\eta$ of the refined grading. For a negative crossing, it is $\frac{e_i + e_{i+1}}{4}$ plus $\eta$ of the refined grading. See Section 4.3 of \cite{OSzNew}.

Figure \ref{GradingFig} also shows Ozsv{\'a}th-Szab{\'o}'s homological gradings for generators labeled $N$, $W$, $E$, and $S$, as well as a modified homological grading which we will relate to the homological grading on $\R_i$ and $\R'_i$.  For a positive crossing, the modified homological grading is the usual grading plus one; for a negative crossing, the modified homological grading is the usual grading minus one.
\end{definition}

The grading structure of $\Pc^i_l$ or $\Nc^i_l$ also includes homomorphisms from the grading groups of the ``input'' and ``output'' algebras $\Ccl(m,1,\emptyset)$ to the grading group of $\Pc^i_l$ or $\Nc^i_l$. For the output algebra, this homomorphism is the identity map on $\Q^{2m}$. For the input algebra, the homomorphism from $\Q^{2m} \to \Q^{2m}$ sends 
\begin{align*}
&\tau_i \mapsto \tau_{i+1} \\
&\beta_i \mapsto\beta_{i+1} \\
&\tau_{i+1} \mapsto \tau_i \\
&\beta_{i+1} \mapsto \beta_i
\end{align*}
and is the identity on all other $\tau_j$ and $\beta_j$. 

Ozsv{\'a}th-Szab{\'o}'s multi-Alexander grading structure may be phrased as follows: $\Pc^i$ and $\Nc^i$ are graded by $\Q^m$, and the homomorphism $\Q^m \to \Q^m$ for the output algebra is the identity. The homomorphism $\Q^m \to \Q^m$ for the input algebra sends $e_i$ to $e_{i+1}$ and vice-versa, and is the identity on all other $e_j$ (this is the homomorphism denoted $\tau_i^{\mathbf{gr}}$ in \cite{OSzNew}; see Equation (5.1) in particular). The squares involving these homomorphisms of grading groups and the map $\eta$ are commutative; thus, the refined grading structure is compatible with Ozsv{\'a}th-Szab{\'o}'s grading structure.

\begin{figure} \centering
\includegraphics[scale=0.625]{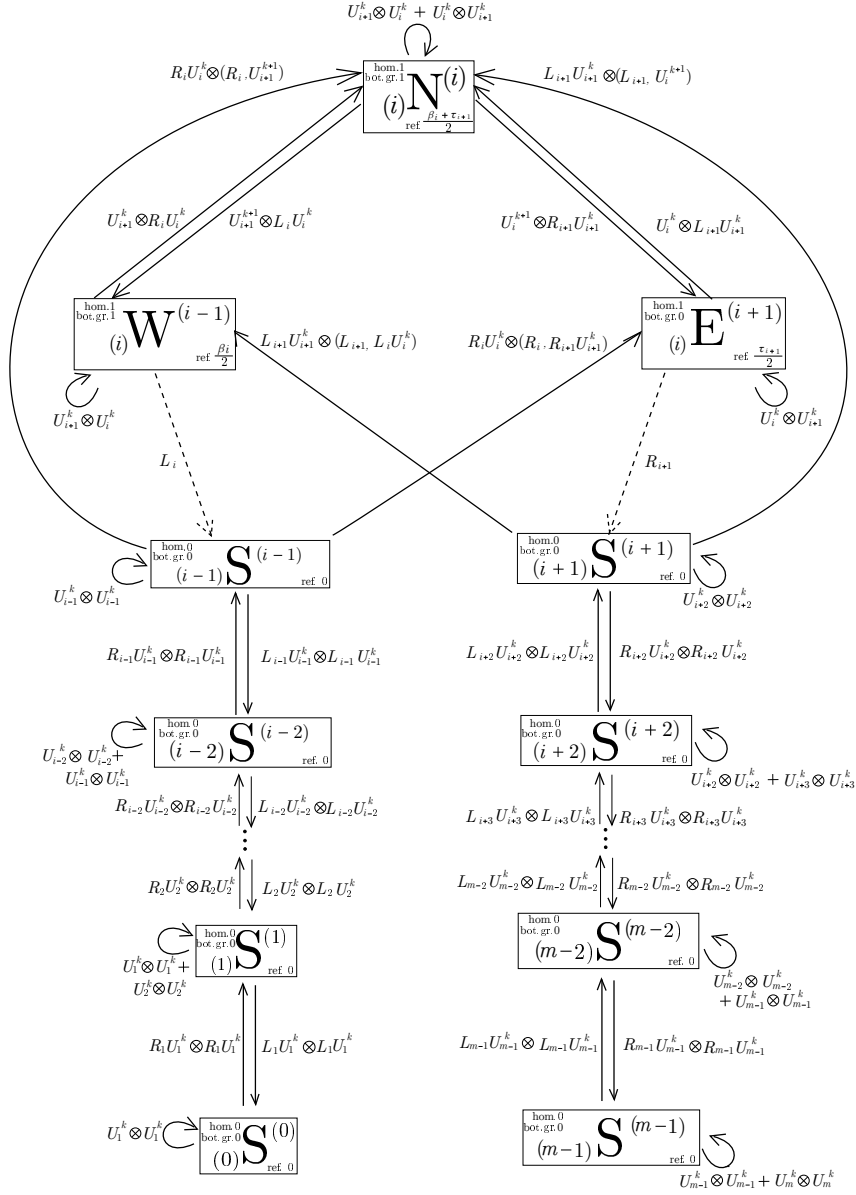}
\caption{The DA bimodule $\Pc^i_l$ for $1 \leq i < m-1$, with grading conventions modified as in Figure~\ref{GradingFig}. We allow all $k \geq 0$, except when the resulting label is $1 \otimes 1$ (each generator should have one self-arrow with this label which we omit from the diagram).}
\label{PiGraphFig}
\end{figure}

\begin{figure} \centering
\includegraphics[scale=0.625]{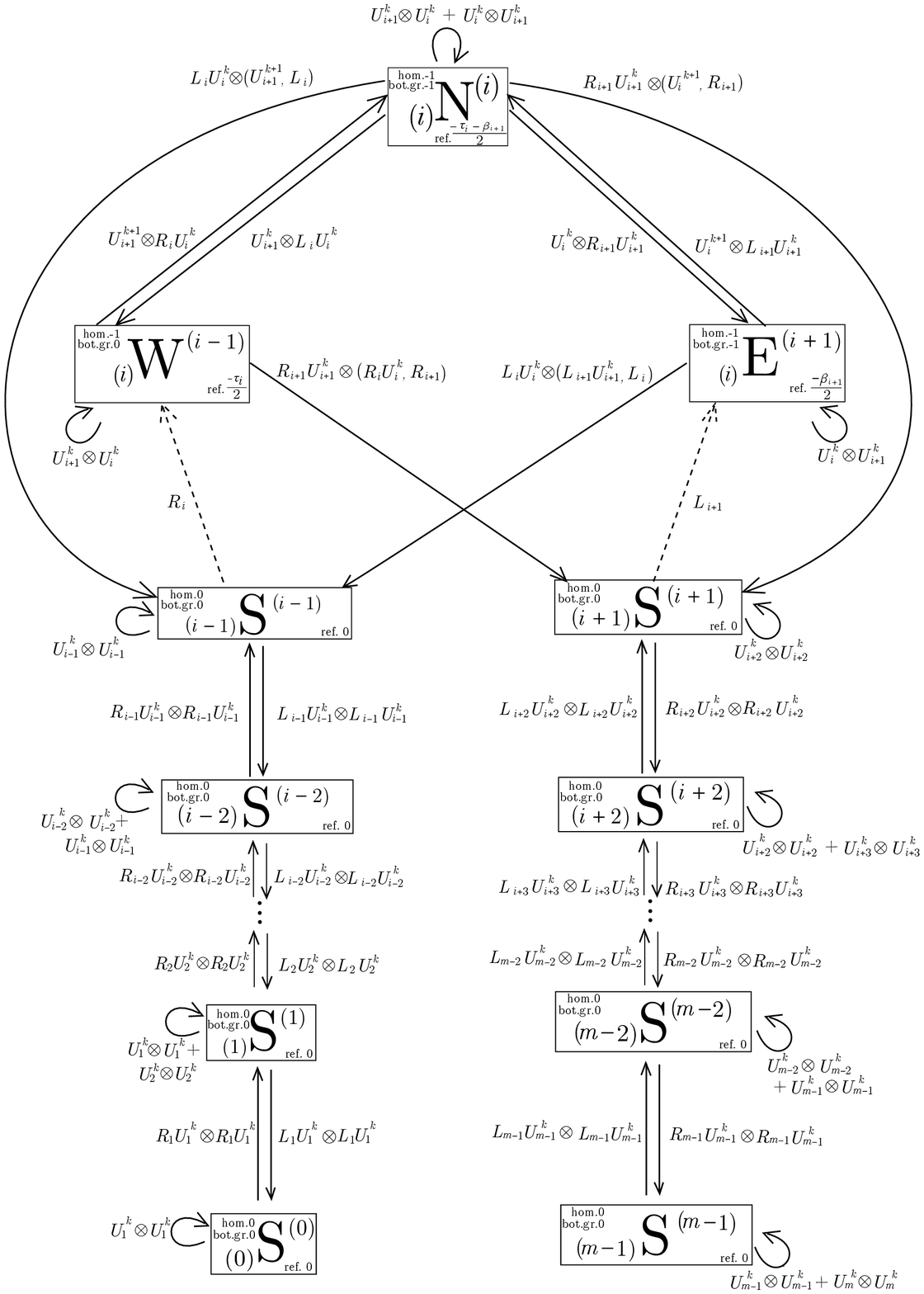}
\caption{The DA bimodule $\Nc^i_l$ for $1 \leq i < m-1$, with grading conventions modified as in Figure~\ref{GradingFig}. The restrictions on $k$ are the same as in Figure~\ref{PiGraphFig}.}
\label{NiGraphFig}
\end{figure}

\begin{definition}[cf. the beginning of Section 5 and Sections 5.1, 5.2 of \cite{OSzNew}, especially Definition 5.7]
Let $1 \leq i < m-1$. The DA bimodule $\Pc^i_l$ is shown in Figure~\ref{PiGraphFig}. The modified homological gradings are displayed in the upper left corner of each DA bimodule generator, and the refined gradings are displayed in the lower right corner. Also shown are the ``bottom gradings'' which come from applying $\epsilon$ to the refined gradings. 
\end{definition}

\begin{definition}[cf. Section 5.5 of \cite{OSzNew}]
Let $1 \leq i < m-1$. The DA bimodule $\Nc^i_l$ is shown in Figure~\ref{NiGraphFig}, with the same grading data as in Figure~\ref{PiGraphFig}. Figure~\ref{NiGraphFig} is obtained from Figure~\ref{PiGraphFig} by reversing all arrows, replacing each $R_j$ with $L_j$ and vice-versa, and for each $\delta^1_3$ arrow, reversing the order of the two inputs. 
\end{definition}

\begin{proposition}
The DA bimodule operations on $\Pc^i_l$ and $\Nc^i_l$ respect the refined multi-Alexander grading.
\end{proposition}

\begin{proof}
This follows from inspection of each arrow in Figure~\ref{PiGraphFig} and Figure~\ref{NiGraphFig}.
\end{proof}

\begin{figure} \centering
\includegraphics[scale=0.625]{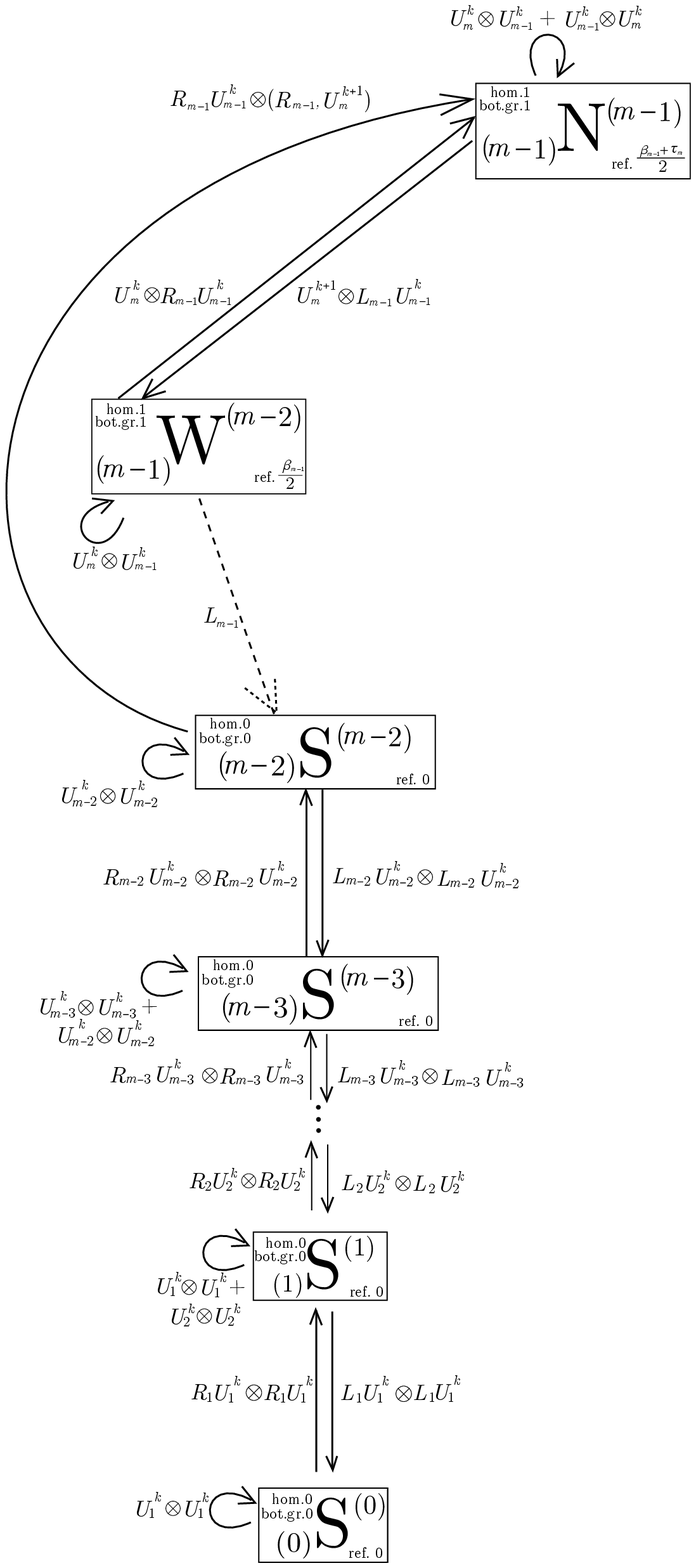}
\caption{The DA bimodule $\Pc^{m-1}_l$, with grading conventions modified from \cite{OSzNew}. We have simply deleted some vertices and their adjacent arrows from Figure~\ref{PiGraphFig}. The bimodule $\Nc^{m-1}_l$ is obtained similarly from Figure \ref{NiGraphFig}.}
\label{SpecialPiGraphFig}
\end{figure}

As with Khovanov-Seidel's construction, the case $i = m-1$ is special. The DA bimodule $\Pc^{m-1}_l$ is shown in Figure \ref{SpecialPiGraphFig}. As before, this diagram is obtained from Figure \ref{PiGraphFig} by deleting some generators and their adjacent arrows. We omit a figure for $\Nc^{m-1}_l$.

\section{Equivalences between bimodules}

\subsection{Restriction and induction of scalars}

\begin{figure} \centering
\includegraphics[scale=0.625]{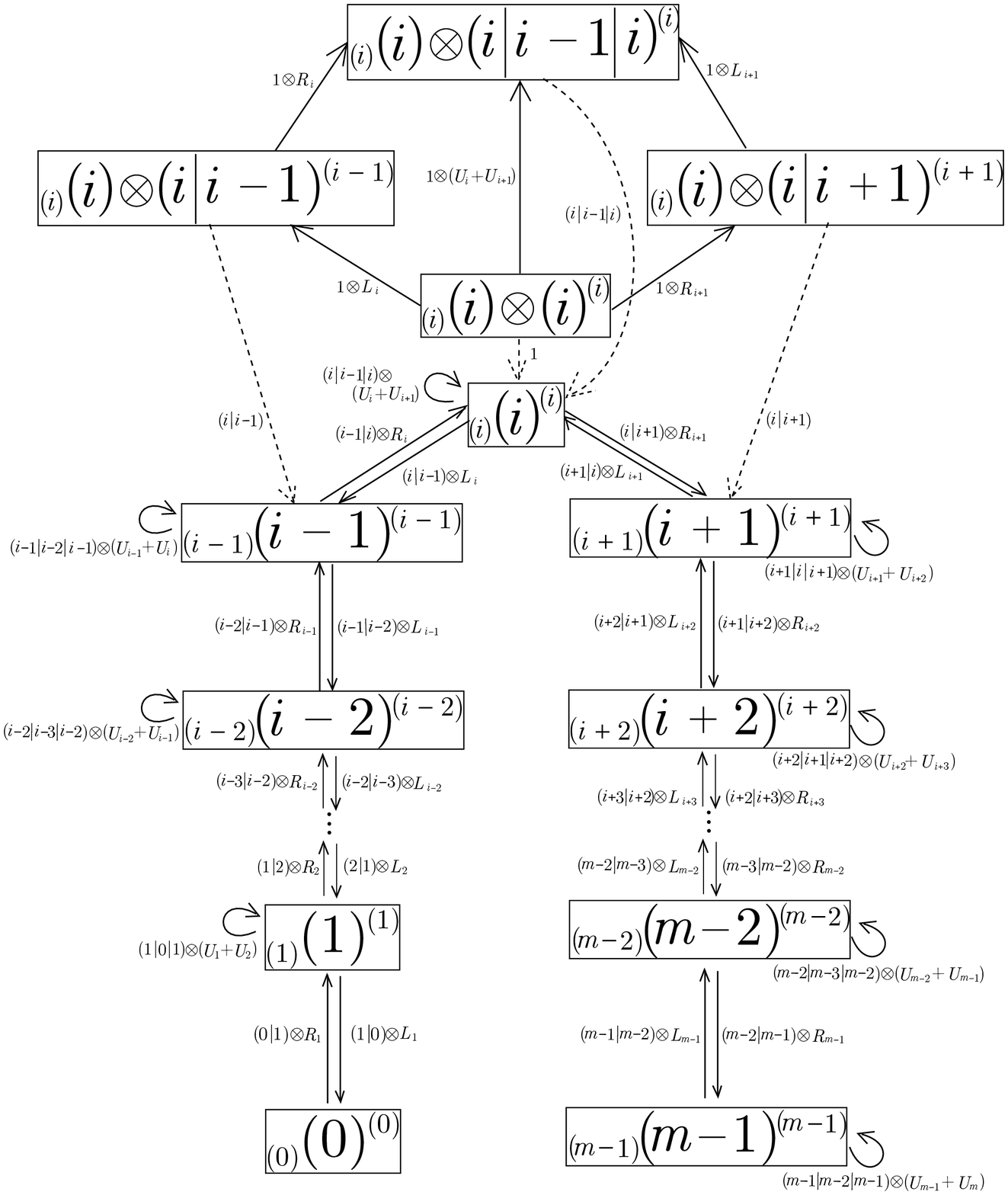}
\caption{The DA bimodule $\Rest_{\phi} \R_i$ for $1 \leq i < m-1$.}
\label{RestRiGraphFig}
\end{figure}

\begin{figure} \centering
\includegraphics[scale=0.625]{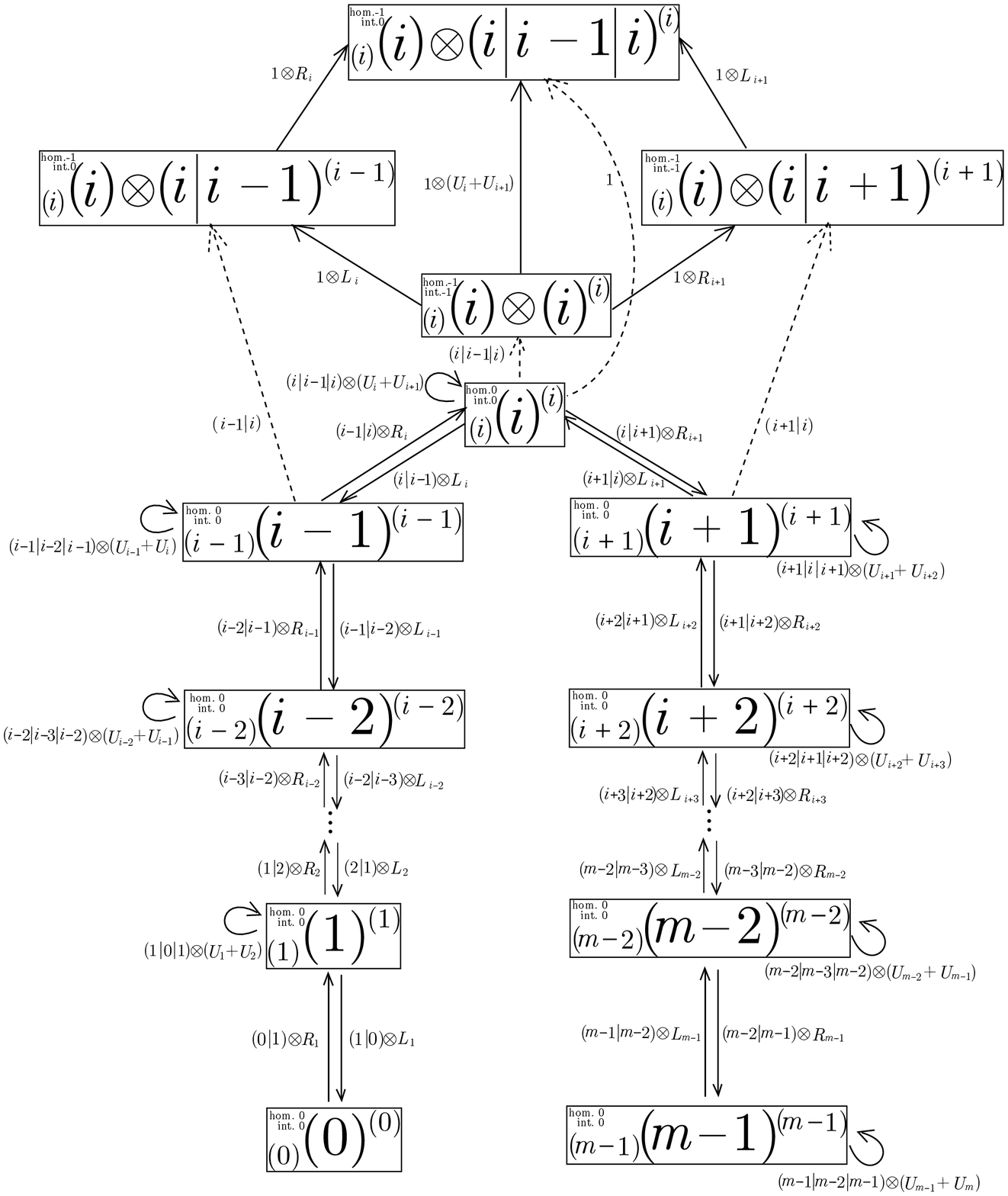}
\caption{The DA bimodule $\Rest_{\phi} \R'_i$ for $1 \leq i < m-1$.}
\label{RestRPrimeiGraphFig}
\end{figure}

Using the map $\phi$, we may obtain DA bimodules over $(A_{m-1} \otimes \Z/2\Z,\Ccl^{bot.gr.}(m,1,\emptyset))$ from the DA bimodules $\R_i$ and $\R'_i$ of Definitions \ref{KSDADef} and \ref{KSDAPrimeDef} by restriction of scalars (see Section 2.4.2 of \cite{LOTBimod}). Following the notation of \cite{LOTBimod},  we call these bimodules $\Rest_{\phi} \R_i$ and $\Rest_{\phi} \R'_i$. In Figure \ref{RestRiGraphFig}, we depict $\Rest_{\phi} \R_i$ in the notation of Figure \ref{RiGraphFig}. Figure \ref{RestRPrimeiGraphFig} shows $\Rest_{\phi} \R'_i$.

\begin{figure} \centering
\includegraphics[scale=0.625]{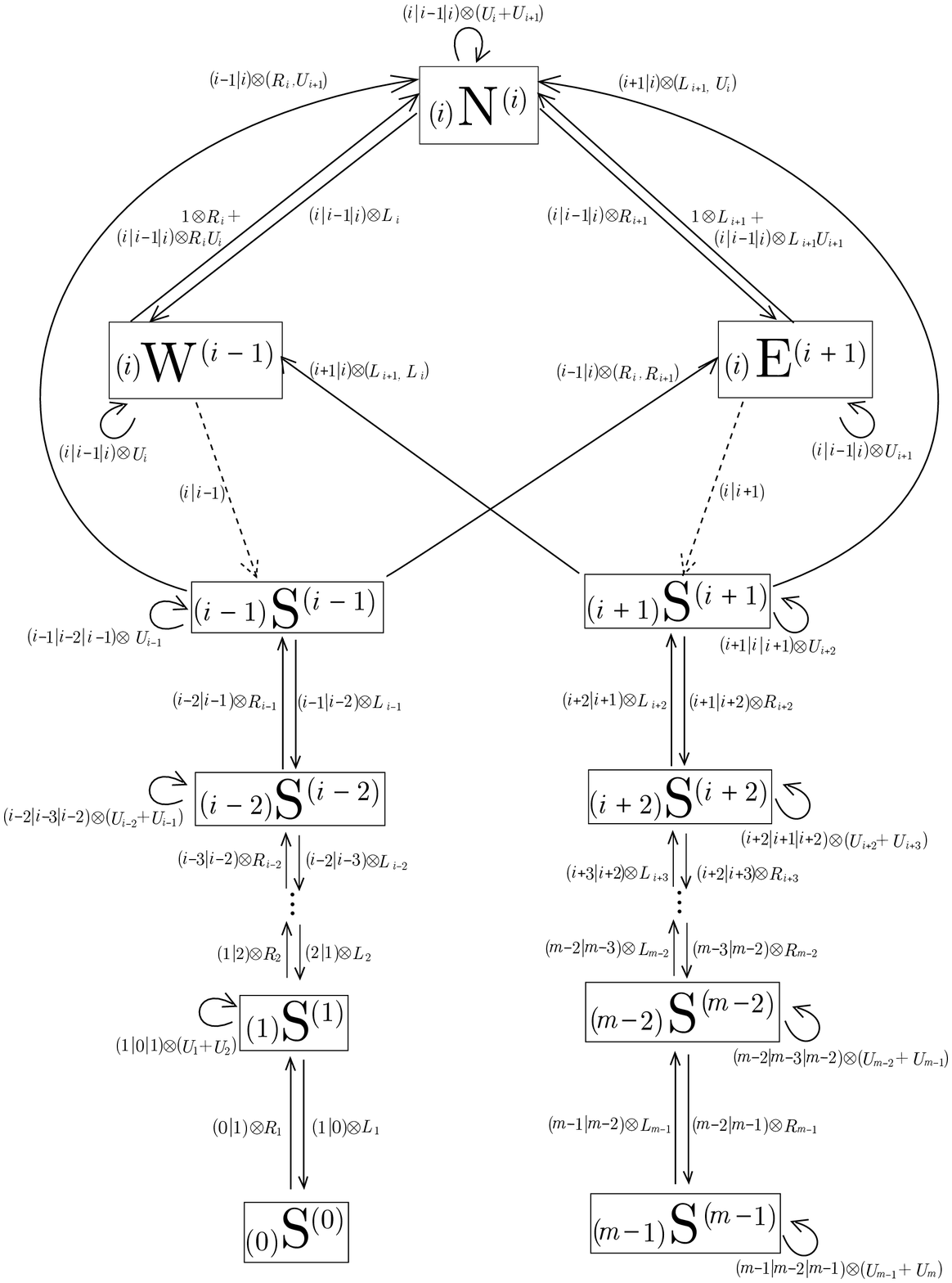}
\caption{The DA bimodule ${^{\phi}}\Induct \Pc^i_l$ for $1 \leq i < m-1$.}
\label{IndPiGraphFig}
\end{figure}

\begin{figure} \centering
\includegraphics[scale=0.625]{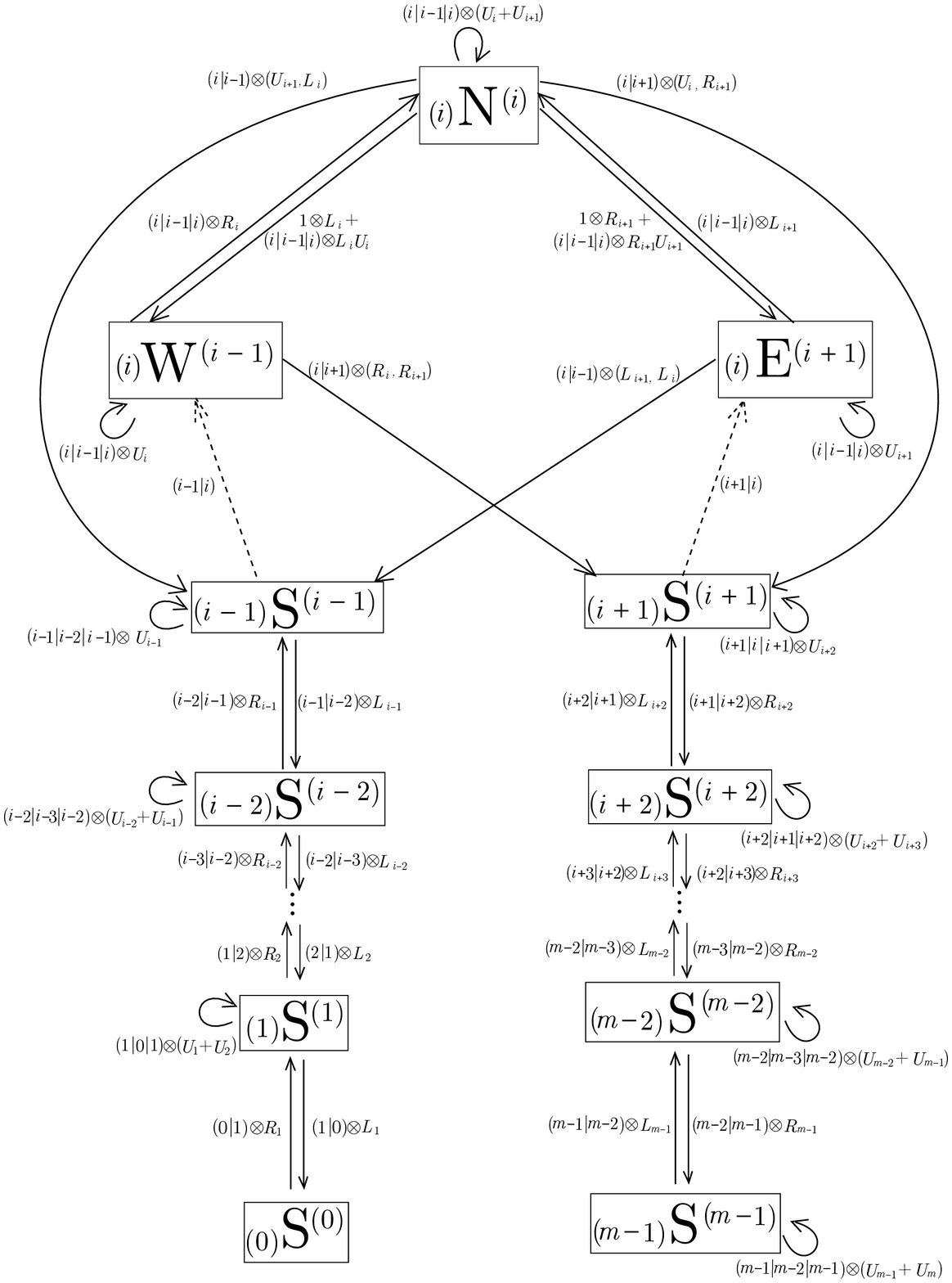}
\caption{The DA bimodule ${^{\phi}}\Induct \Nc^i_l$ for $1 \leq i < m-1$.}
\label{NegIndPiGraphFig}
\end{figure}

We may also obtain DA bimodules over $(A_{m-1} \otimes \Z/2\Z, \Ccl^{bot.gr.}(m,1,\emptyset))$ from the DA bimodules $\Pc^i_l$ and $\Nc^i_l$ by induction of scalars. Following the notation of \cite{LOTBimod}, Section 2.4.2, we denote these bimodules by ${^{\phi}}\Induct \Pc^i_l$ and ${^{\phi}}\Induct \Nc^i_l$. Figure \ref{IndPiGraphFig} shows ${^{\phi}}\Induct \Pc^i_l$, and Figure \ref{NegIndPiGraphFig} shows ${^{\phi}}\Induct \Nc^i_l$. 

\subsection{A homotopy equivalence}
Here we show that $\Rest_{\phi} \R_i$ and ${^{\phi}}\Induct \Pc^i_l$, as well as $\Rest_{\phi} \R'_i$ and ${^{\phi}}\Induct \Nc^i_l$, are homotopy equivalent as DA bimodules.

\begin{figure} \centering
\includegraphics[scale=0.625]{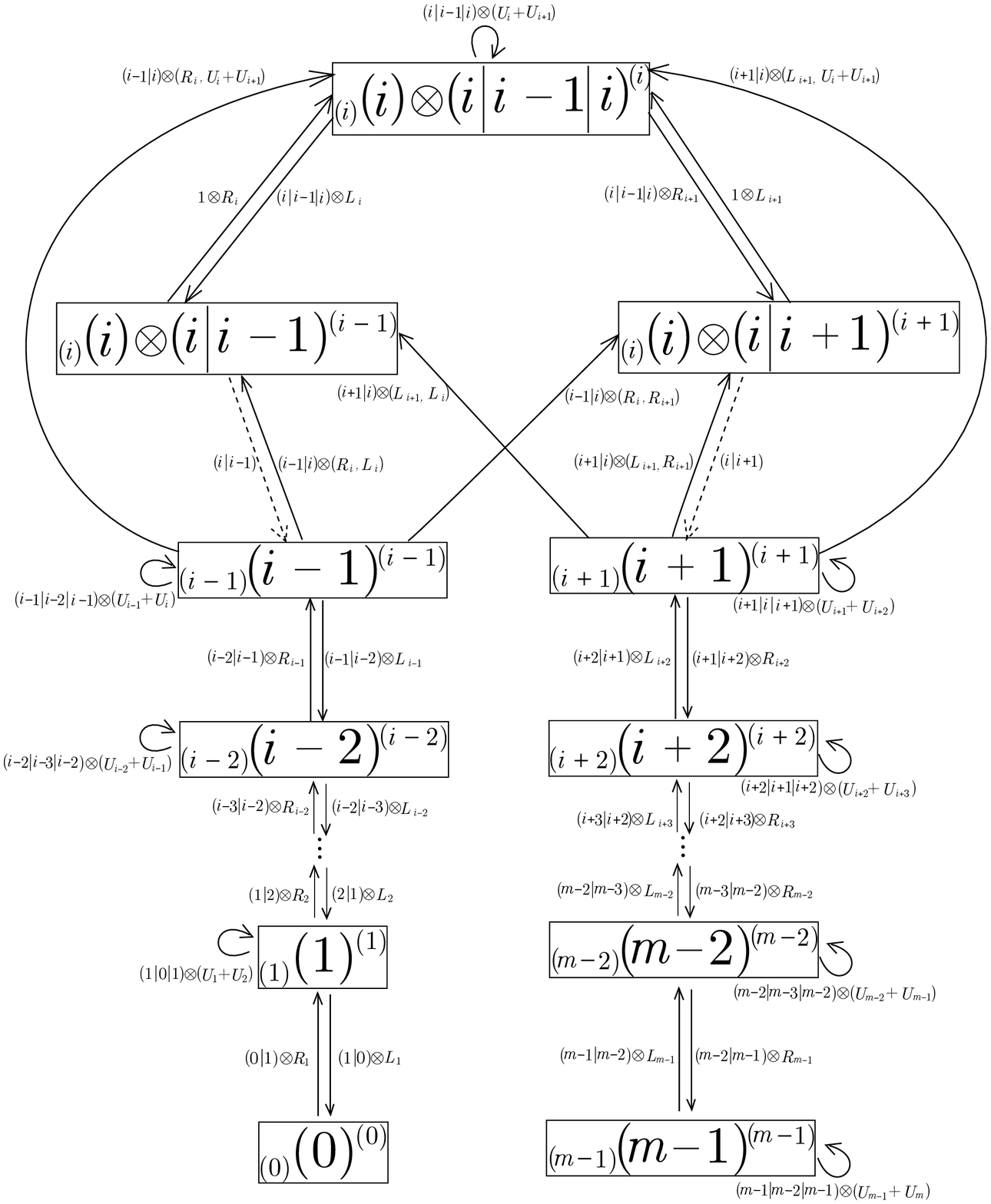}
\caption{A DA bimodule homotopy equivalent to $\Rest_{\phi} \R_i$ for $1 \leq i < m-1$.}
\label{HtpyRestRiGraphFig}
\end{figure}

\begin{figure} \centering
\includegraphics[scale=0.625]{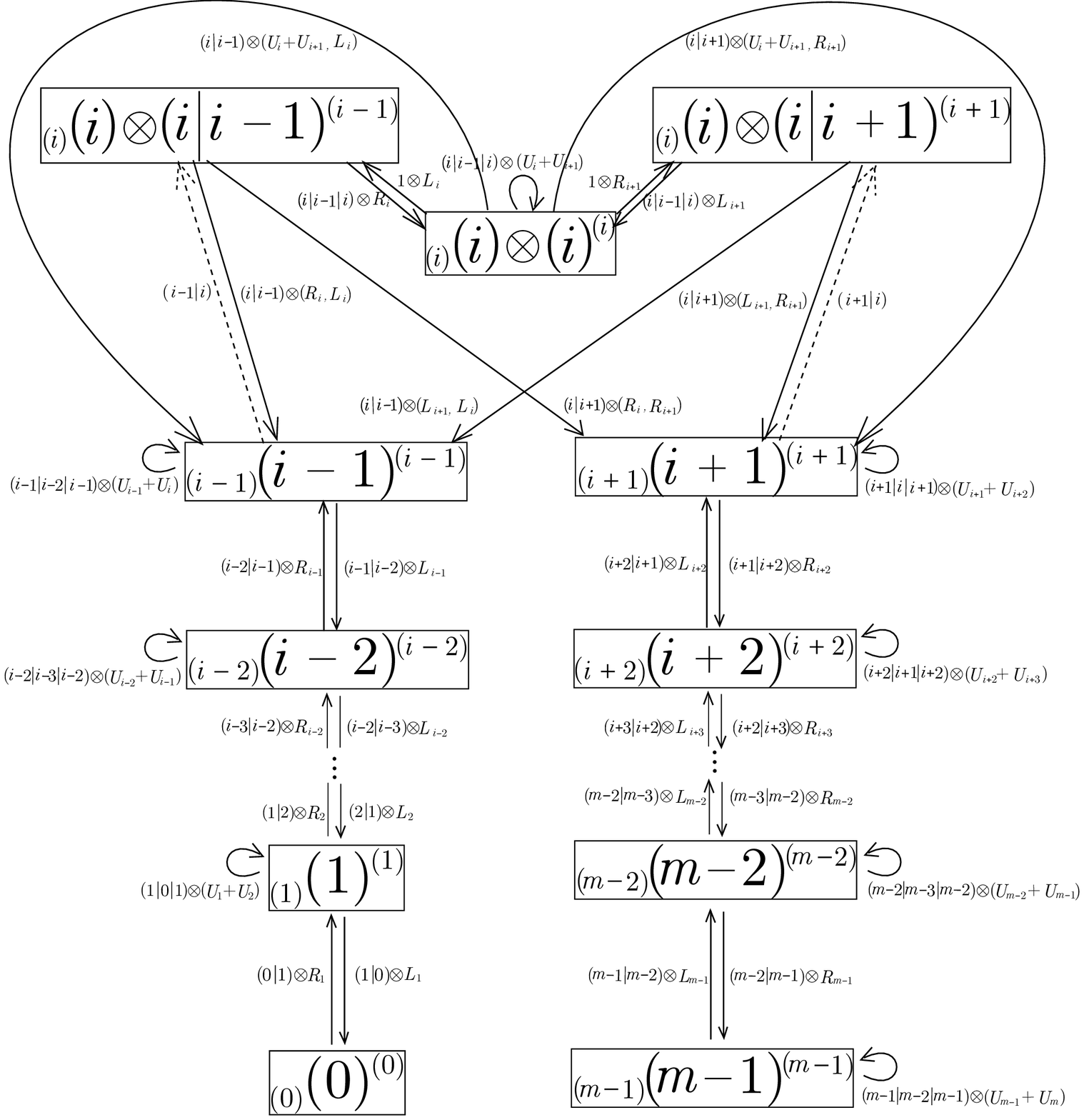}
\caption{A DA bimodule homotopy equivalent to $\Rest_{\phi} \R'_i$ for $1 \leq i < m-1$.}
\label{HtpyRestRPrimeiGraphFig}
\end{figure}

\begin{lemma}\label{HtpyLemma} For $1 \leq i < m-1$, $\Rest_{\phi} \R_i$ is homotopy equivalent to the DA bimodule shown in Figure \ref{HtpyRestRiGraphFig}.
\end{lemma}

\begin{proof}
We use Lemma 2.12 of \cite{OSzNew}. The DA bimodule $\Rest_{\phi} \R_i$ is strictly unital, because $\delta^1_2(\cdot,1) = \id$ and there are no $\delta^1_j$ actions for $j > 2$. As a Type D structure, $\Rest_{\phi} \R_i$ may be described by forgetting the solid arrows in Figure~\ref{RestRiGraphFig}, resulting in a greatly simplified diagram. The Type D structure $Z$ in Ozsv{\'a}th-Szab{\'o}'s Lemma 2.12 is obtained by discarding the generators $\langle (i) \otimes (i) \rangle$ and $\langle (i) \rangle$ in Figure~\ref{RestRiGraphFig} (as in Section~\ref{DAReformSect}, we use angle brackets for DA bimodule generators to differentiate them from algebra elements). The homomorphism $g$ in Lemma 2.12 is the projection from $\Rest_{\phi} \R_i$ onto the remaining generators. 

The homomorphism $f$ is the inclusion into $\Rest_{\phi} \R_i$ on all generators of $Z$ except $\langle (i) \otimes (i|i-1|i) \rangle$ . We define
\[
f\langle (i) \otimes (i|i-1|i) \rangle  = 1 \otimes \langle (i) \otimes (i|i-1|i) \rangle + (i|i-1|i) \otimes \langle (i) \otimes (i) \rangle.
\]
Then $f$ is a valid homomorphism of Type D structures and preserves the bigradings. We have $g \circ f = \id_Z$. Also, $f \circ g = \id_{\Rest_{\phi} \R_i} + dT$, where 
\[
T: \Rest_{\phi} \R_i \to \Rest_{\phi} \R_i
\]
sends $\langle (i) \rangle$ to $\langle (i) \otimes (i) \rangle$ and sends all other generators to zero. Since $T^2 = 0$, all conditions of Lemma 2.12 are satisfied.

We conclude that $\Rest_{\phi} \R_i$ is homotopy equivalent, as a DA bimodule, to a DA structure on $Z$ which is described explicitly in the proof of Lemma 2.12 of \cite{OSzNew}. In terms of the graphical representation, we look for patterns of arrows
\[
x \rightarrow f(x) \xrightarrow{\alpha} y \rightarrow g(y),
\]
\[
x \rightarrow f(x) \xrightarrow{\alpha_1} \langle (i) \rangle \xrightarrow{T} \langle (i) \otimes (i) \rangle \xrightarrow{\alpha_2} y \rightarrow g(y),
\]
etc. in Figure~\ref{RestRiGraphFig} (only these two patterns actually occur). We require $\alpha_1$ and $\alpha_2$ to be solid arrows representing $\delta^1_2$ operations ($\Rest_{\phi} \R_i$ has no $\delta^1_j$ operations for $j > 2$); $\alpha$ may represent any type of DA operation.

The first arrow pattern gives us all arrows of Figure \ref{RestRiGraphFig} which do not involve vertices $\langle (i) \otimes (i) \rangle$ or $\langle (i) \rangle$, as well as three additional solid arrows representing $\delta^1_2$ operations and coming from the ``extra'' term of $f$ which is not just the inclusion of generators of $Z$ into $\Rest_{\phi} \R_i$. The second pattern gives us six additional solid arrows which represent $\delta^1_3$ operations. The resulting diagram is Figure~\ref{HtpyRestRiGraphFig}, proving the lemma.
\end{proof}

\begin{remark}
A similar statement holds when $i = m-1$, with the same proof. To save space, we will only give the details for the generic case $i < m-1$.
\end{remark}

Analogously, we have:
\begin{lemma}\label{PrimeHtpyLemma} For $1 \leq i < m-1$, $\Rest_{\phi} \R'_i$ is homotopy equivalent to the DA bimodule shown in Figure \ref{HtpyRestRPrimeiGraphFig}.
\end{lemma}

\begin{proof}
The proof closely follows Lemma~\ref{HtpyLemma}. Rather than $Z$, we have a Type D structure $Z'$ obtained by discarding the generators $\langle (i) \otimes (i|i-1|i) \rangle$ and $\langle (i) \rangle$ from Figure \ref{RestRPrimeiGraphFig}. This time, the homomorphism $f$ from $Z'$ to $\Rest_{\phi} \R'_i$ is the inclusion, and the homomorphism $g$ from $\Rest_{\phi} \R'_i$ to $Z'$ is the usual projection on all generators except $\langle (i) \otimes (i|i-1|i) \rangle$. We have
\[
g(\langle (i) \otimes (i|i-1|i) \rangle) := (i|i-1|i) \otimes \langle (i) \otimes (i) \rangle.
\]
The map $T$ sends $\langle (i) \otimes (i|i-1|i) \rangle$ to $\langle (i) \rangle$ and sends all other generators to zero. The rest of the proof is the same as for Lemma~\ref{HtpyLemma}, except that $g$ (rather than $f$) has the ``extra'' term.
\end{proof}
Again, a similar statement holds when $i = m-1$.

Below, we will need to work with homomorphisms of DA bimodules explicitly. Let $X$ and $Y$ be DA bimodules over dg algebras $(\A,\A')$. We recall (see Definition 2.2.43 of \cite{LOTBimod}) that a homomorphism $f: X \to Y$ of DA bimodules consists of maps 
\[
f_j: X \otimes \A'_+[1]^{j-1} \to \A \otimes Y
\]
for $j \geq 1$, satisfying compatibility conditions (recall that $[\cdot]$ denotes upward shift in the homological grading). Here, $\A'_+$ denotes the augmentation ideal of $\A'$; the algebras considered in this paper have natural augmentations (the projections onto the idempotent rings which send all quiver generators to zero).

\begin{figure} \centering
\includegraphics[scale=0.625]{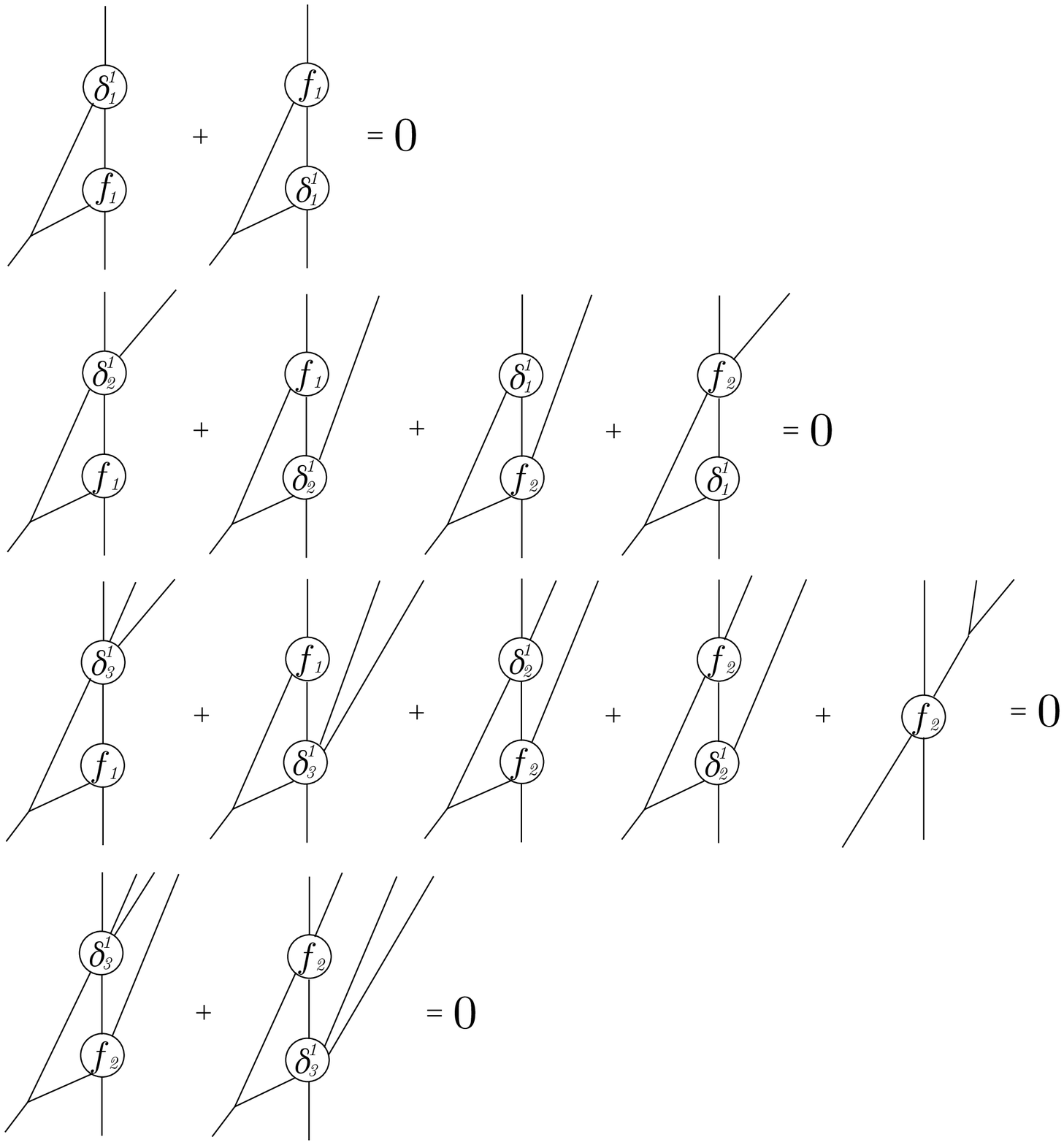}
\caption{DA bimodule homomorphism compatibility conditions, in the special case of Lemma~\ref{DAConditionsLemma}.}
\label{DAMorRelsFig}
\end{figure}

\begin{lemma}\label{DAConditionsLemma}
Suppose $\A$ and $\A'$ are dg algebras with zero differential, the DA bimodules $X$ and $Y$ have $\delta^1_j = 0$ for $j \geq 4$, and the morphism $f$ has $f_j = 0$ for $j \geq 3$. Then the DA compatibility conditions for $f$ (i.e. $\partial(f) = 0$ for the differential $\partial$ defined in Figure 2 of \cite{LOTBimod}) are equivalent to the conditions shown in Figure~\ref{DAMorRelsFig}. 
\end{lemma}

\begin{proof} This follows from unpacking Lipshitz-Ozsv{\'a}th-Thurston's definitions in this special case. For an explanation of the notation of Figure~\ref{DAMorRelsFig}, see Definition 2.2.43 of \cite{LOTBimod}, which contains the Figure 2 referred to above.
\end{proof}

\begin{theorem}[cf. Theorem~\ref{IntroBimodThm}]\label{BodyBimodThmP}
$\Rest_{\phi} \R_i$ and ${^{\phi}}\Induct \Pc^i_l$ are homotopy equivalent as DA bimodules over $(A_{m-1} \otimes \Z/2\Z, \Ccl^{bot.gr.}(m,1,\emptyset))$.
\end{theorem}

\begin{proof}
We will assume $i < m-1$ (the case $i = m-1$ is similar enough that we omit it for brevity). By Lemma~\ref{HtpyLemma}, it suffices to show that the DA bimodules shown in Figures \ref{IndPiGraphFig} and \ref{HtpyRestRiGraphFig} are homotopy equivalent; in fact, we will show they are isomorphic. Let $Z$ denote the DA bimodule of Figure~\ref{HtpyRestRiGraphFig}.

The shape of Figures \ref{IndPiGraphFig} and \ref{HtpyRestRiGraphFig} suggests natural maps $\iota$ from the generators of $Z$ to the generators of ${^{\phi}}\Induct \Pc^i_l$ and $\iota'$ in the other direction. We will define additional maps $h,h'$ such that 
\[
\iota + h: Z \to {^{\phi}}\Induct \Pc^i_l
\]
is an isomorphism of DA bimodules with inverse
\[
\iota' + h': {^{\phi}}\Induct \Pc^i_l \to Z.
\]

The map $h$ will only be nonzero on the generators $\langle (i-1) \rangle$ and $\langle (i+1) \rangle$ of $Z$. We define $h_1 = 0$ and
\begin{align*}
&h_2(\langle (i-1) \rangle, U_i) = (i-1|i)\otimes W;\\
&h_2(\langle (i-1) \rangle, R_iU_i) = (i-1|i) \otimes N;\\
&h_2(\langle (i+1) \rangle, U_{i+1}) = (i+1|i) \otimes E;\\
&h_2(\langle (i+1) \rangle, L_{i+1}U_{i+1}) = (i+1|i) \otimes N.\\
\end{align*}

$h'$ is only nonzero on the generators ${_{(i-1)}}S^{(i-1)}$ and ${_{(i+1)}}S^{(i+1)}$ of ${^{\phi}}\Induct \Pc^i_l$. We define $h'_1 = 0$ and 
\begin{align*}
&h'_2({_{(i-1)}}S^{(i-1)}, U_i) = (i-1|i) \otimes \langle (i) \otimes (i|i-1) \rangle;\\
&h'_2({_{(i-1)}}S^{(i-1)}, R_iU_i) = (i-1|i) \otimes \langle (i) \otimes (i|i-1|i) \rangle;\\ 
&h'_2({_{(i+1)}}S^{(i+1)}, U_{i+1}) = (i+1|i) \otimes \langle (i) \otimes (i|i+1) \rangle;\\
&h'_2({_{(i+1)}}S^{(i+1)}, L_{i+1}U_{i+1}) = (i+1|i) \otimes \langle (i) \otimes (i|i-1|i) \rangle.\\
\end{align*}
Note that $h$ and $h'$ have the desired properties with respect to the homological and intrinsic gradings (recall that ${^{\phi}}\Induct \Pc^i_l$ is given the ``bottom gradings'').

We first check that $\iota + h$ is a valid DA bimodule homomorphism using Lemma~\ref{DAConditionsLemma}. We have $(\iota + h)_1 = \iota$ and $(\iota + h)_2 = h_2$.

The relation on the first row of Figure~\ref{DAMorRelsFig} holds because the dotted arrows are the same in Figures \ref{IndPiGraphFig} and \ref{HtpyRestRiGraphFig}. The relation on the second row of Figure~\ref{DAMorRelsFig} can be checked by comparing the solid arrows representing $\delta^1_2$ actions in Figures \ref{IndPiGraphFig} and \ref{HtpyRestRiGraphFig}. The differences between these arrows in the two figures (there are $6$ such differences) are accounted for using $h_2$ and the dotted arrows.

The relation on the third row of Figure~\ref{DAMorRelsFig} can be checked by, first, comparing the solid arrows representing $\delta^1_3$ actions in Figures \ref{IndPiGraphFig} and \ref{HtpyRestRiGraphFig}. The differences (there are $4$ this time) are accounted for using terms of the rightmost type on row three of Figure~\ref{DAMorRelsFig} (``multiplication terms''). These terms arise because we can write the algebra inputs of $h_2$ and $h'_2$ as $U_i = R_i \cdot L_i$, $R_i U_i = R_i \cdot U_i = U_i \cdot R_i$, $U_{i+1} = L_{i+1} \cdot R_{i+1}$, and $L_{i+1}U_{i+1} = L_{i+1} \cdot U_{i+1} = U_{i+1} \cdot L_{i+1}$. 

All multiplication terms except those from $R_i U_i = U_i \cdot R_i$ and $L_{i+1} U_{i+1} = U_{i+1} \cdot L_{i+1}$ cancel with terms of the leftmost or second-leftmost type on row three of Figure~\ref{DAMorRelsFig}. The remaining two multiplication terms cancel with terms of the third- or second-rightmost type on this row. Inspection of the diagrams shows that these are all the nontrivial relations from the third row of Figure~\ref{DAMorRelsFig}.

Finally, the relation on the bottom row of Figure~\ref{DAMorRelsFig} holds because both of its terms are zero individually. Analogous reasoning shows that $\iota' + h'$ is also a DA bimodule homomorphism.

The composition of DA bimodule morphisms is defined, along with the differential, in Figure 2 of \cite{LOTBimod}. We may expand out the composition $(\iota' + h') \circ (\iota + h)$ as
\[
\id_Z + h' \circ \iota + \iota' \circ h + h' \circ h.
\]
Similarly, $(\iota + h) \circ (\iota' + h')$ is
\[
\id_{{^{\phi}}\Induct \Pc^i_l} + h \circ \iota' + \iota \circ h' + h \circ h'.
\]
One can check that $h' \circ \iota = \iota' \circ h$ and $h \circ \iota' = \iota \circ h'$, as well as $h' \circ h = 0$ and $h \circ h' = 0$. Thus, $\iota + h$ and $\iota' + h'$ are inverses.

\end{proof}

\begin{theorem}[cf. Theorem~\ref{IntroBimodThm}]\label{BodyBimodThmN}
$\Rest_{\phi} \R'_i$ and ${^{\phi}}\Induct \Nc^i_l$ are homotopy equivalent as DA bimodules over $(A_{m-1} \otimes \Z/2\Z, \Ccl^{bot.gr.}(m,1,\emptyset))$.
\end{theorem}

\begin{proof}
As with Lemma~\ref{PrimeHtpyLemma}, we will only consider the generic case $1 \leq i < m-1$ for simplicity, and we only outline the differences with the proof of Theorem \ref{BodyBimodThmP}. Now the map $h$ is defined by $h_1 = 0$ and 
\begin{align*}
&h_2(\langle (i) \otimes (i|i-1) \rangle, U_i) = (i|i-1) \otimes {_{(i-1)}}S^{(i-1)} \\
&h_2(\langle (i) \otimes (i) \rangle, L_iU_i) = (i|i-1) \otimes {_{(i-1)}}S^{(i-1)} \\
&h_2(\langle (i) \otimes (i|i+1) \rangle, U_{i+1}) = (i|i+1) \otimes {_{(i+1)}}S^{(i+1)} \\
&h_2(\langle (i) \otimes (i) \rangle, R_{i+1}U_{i+1}) = (i|i+1) \otimes {_{(i+1)}}S^{(i+1)}. \\
\end{align*}
The map $h'$ is defined by $h'_1 = 0$ and
\begin{align*}
&h'_2(W, U_i) = (i|i-1) \otimes \langle (i-1) \rangle \\
&h'_2(N, L_iU_i) = (i|i-1) \otimes \langle (i-1) \rangle \\
&h'_2(E, U_{i+1}) = (i|i+1) \otimes \langle (i+1) \rangle \\
&h'_2(N, R_{i+1}U_{i+1}) = (i|i+1) \otimes \langle (i+1) \rangle. \\
\end{align*}
As above, the DA homomorphism consistency conditions of Lemma \ref{DAConditionsLemma} are checked using Figure \ref{NegIndPiGraphFig} and Figure \ref{HtpyRestRPrimeiGraphFig}. 
\end{proof}

\bibliographystyle{alpha}
\bibliography{biblio}

\end{document}